\documentclass[11pt,reqno]{amsart}
\usepackage{amsmath,amssymb,mathrsfs,amscd}
\usepackage{graphicx,subfigure,cite}
\usepackage{color}
\usepackage [section]{placeins}
\newtheorem{theorem}{Theorem}[section]

\newtheorem{lemma}[theorem]{Lemma}

\newtheorem{remark}[theorem]{Remark}

\newtheorem{definition}{Definition}[section]

\numberwithin{equation}{section}

\allowdisplaybreaks[4]
\title[an inverse potential problem]{An inverse potential problem for the stochastic diffusion equation with a multiplicative white noise}

\author{Xiaoli Feng}
\address{School of Mathematics and Statistics, Xidian University, Xi'an, 713200, P. R. China}
\email{xiaolifeng@xidian.edu.cn}

\author{Peijun Li}
\address{Department of Mathematics, Purdue University, West Lafayette, Indiana 47907, USA}
\email{lipeijun@math.purdue.edu}

\author{Xu Wang}
\address{LSEC, ICMSEC, Academy of Mathematics and Systems Science, Chinese Academy of Sciences, Beijing 100190, China, and School of Mathematical Sciences, University of Chinese Academy of Sciences, Beijing 100049, China}
\email{wangxu@lsec.cc.ac.cn}

\thanks{The first author is supported by the Natural Science Basic Research Program of Shaanxi (No. 2023-JC-YB-054). The second author is supported in part by the NSF grant DMS-2208256. The third author is supported by the NNSF of China (Nos. 11971470 and 11871068).}

\subjclass[2010]{35K05, 35R30, 60H15, 80A23}

\keywords{stochastic diffusion equation, inverse potential problem, multiplicative white noise, mild solution, uniqueness, regularization}

\date{}

\begin{document}
\maketitle

\begin{abstract}
This work concerns the direct and inverse potential problems for the stochastic diffusion equation driven by a multiplicative time-dependent white noise. The direct problem is to examine the well-posedness of the stochastic diffusion equation for a given potential, while the inverse problem is to determine the potential from the expectation of the solution at a fixed observation point inside the spatial domain. The direct problem is shown to admit a unique and positive mild solution if the initial value is nonnegative. Moreover, an explicit formula is deduced to reconstruct the square of the potential, which leads to the uniqueness of the inverse problem for nonnegative potential functions. Two regularization methods are utilized to overcome the instability of the numerical differentiation in the reconstruction formula. Numerical results show that the methods are effective to reconstruct both smooth and nonsmooth potential functions.
\end{abstract}

\section{Introduction}

As a basic and important mathematical model, the diffusion equation has been used to describe many physical, biological, chemical, economic, and social phenomena. The inverse diffusion problem aims to determine unknown parameters in the equation by using some measured data. It has significant applications in gas dynamics, chemical kinetics, biophysics, medicine, ecology, finance, and many other sciences. Due to the applied and mathematical interests, the inverse problem for the diffusion equation is one of the most studied problems in the inverse problem community. There is a considerable amount of information regarding their solutions \cite{IsakovIPDE}. This work is concerned with an inverse potential problem for the diffusion equation.

The topic of inverse potential problems for the diffusion equation has been extensively explored from both the mathematical and numerical aspects. For example, the uniqueness and stability results are available in \cite{Rundell1987, Canuto2001, Tuan2018, YangCF2021} and \cite{Kian2013}, respectively; some of the computational studies can be found in \cite{Cannon1994, Ou2011, Bell2013}. During the past two decades, the fractional diffusion equation has received much attention in applied disciplines since it can capture more faithfully the dynamics of anomalous diffusion processes. Correspondingly, it has become an area of intensive research on inverse problems for the fractional diffusion equation. In \cite{Chengj2013}, the authors considered the uniqueness of simultaneously recovering the fractional order and the space-dependent diffusion coefficient. The results of global uniqueness were given in \cite{Kian2018} on the inversion of the density, conductivity, and potential functions. The stability estimate was obtained in \cite{Kian2013} for a time-dependent potential function in cylindrical domains. Various reconstruction methods were also developed for the numerical solutions, such as the iterative Newton-type method \cite{Rundell2019}, the modified optimal perturbation method \cite{Ligs2013}, the fixed point iteration method \cite{JinbtZhouz21}, and the Levenberg--Marquardt method \cite{JiangWei2021}. Related results on uniqueness and stability were also discussed in the above literature.

Recently, a lot of attention has been paid to inverse problems of the stochastic diffusion equation, one of the fundamental models in stochastic partial differential equations. Stochastic inverse problems are more challenging than their deterministic counterparts due to uncertainties and randomness. They are much less studied but remain topics of much ongoing research. In \cite{Yuan2017}, a conditional stability was obtained to determine the initial data for the stochastic parabolic equation. In \cite{Lv2012}, the author considered the backward and inverse source problems for the stochastic parabolic equations. The inverse random source problems were studied in \cite{Niu2020, FLW2020, GLWX2021} for the stochastic time-fractional diffusion equations. The inverse potential problem was examined in \cite{Ruan2021} for the diffusion equation with a random source. These problems are based on the stochastic diffusion equations with an additive noise. There is no result for the inverse potential problem of the stochastic diffusion equation with a multiplicative noise. For the first time, we address this problem in this work.

Specifically, we consider the initial boundary value problem for the stochastic diffusion equation driven by a multiplicative white noise
\begin{equation}\label{eq:model}
\left\{
\begin{aligned}
\frac{\partial u}{\partial t}(x,t)=&~\Delta u(x,t)+q(t)u(x,t){\dot{B}}(t), &&  (x,t)\in{D}\times(0,T],\\
u(x,t)=&~0, && (x,t)\in{\partial D}\times(0,T], \\
u(x,0)=&~u_0(x), && x\in\overline{D},
\end{aligned}
\right.
\end{equation}
where $D\subset \mathbb{R}^d$ is a bounded and open domain with Lipschitz boundary $\partial D$, the potential $q\in L^2(0,T)$ is assumed to be a deterministic time-dependent function, and the initial value $u_0$ is a deterministic and nonnegative function. Here, the white noise $\dot B$ is the formal derivative of the standard Brownian motion $B$ and the multiplicative noise $q(t)u(x,t)\dot B(t)$ holds in the It\^o integral sense. In this  model, $q(t)\dot B(t)$ can be viewed as a random potential depending on the time, where $q$ represents the strength of the randomness. The direct problem is to examine the existence, uniqueness, and regularity of the solution $u$ to \eqref{eq:model} for a given potential function $q$. The inverse problem is to determine the potential function $q$ from the measured data $\{u(x_*, t)\}_{t\in(0, T]}$ at some interior observation point $x_*\in D$. In addition to being nonlinear, the inverse potential problem is ill-posed.

In this paper, we study both the direct and inverse problems. It is worth mentioning that the stochastic diffusion equation in \eqref{eq:model} should be interpreted as a stochastic integral equation due to the roughness of the white noise. Using the theory of semigroup and stochastic partial differential equations, we show that the direct problem admits a unique mild solution and the mild solution is an analytically strong solution if the initial condition $u_0$ satisfies some regularity condition. Moreover, an analytical solution is constructed by using the eigenfunctions of the Laplacian. The analytical solution is shown to be strictly positive if the initial condition $u_0$ is nonnegative and not identically zero. Based on the analytical solution and the relation between the deterministic and stochastic diffusion equations, an explicit formula is deduced to reconstruct $q^2$ by the averaged data $\{\mathbb E[\ln u(x_*,t)]\}_{t\in[0,T]}$ over the sample space. As a byproduct of the explicit formula, the uniqueness is obtained to recover $q^2$, which further implies the uniqueness of the inverse problem if the potential function is nonnegative. However, it is unstable to directly make use of the explicit formula to reconstruct $q^2$ numerically since it involves the temporal derivative of the data. Two regularization methods, the Tikhonov method and the spectral cut-off method, are employed to to handle the instability issue. Based on a periodic extension of the data function, the numerical differentiation is implemented efficiently via the fast Fourier transform. Numerical experiments are carried out to investigate the influence of various parameters on the reconstructions. The numerical results show that the methods are effective for both smooth and nonsmooth examples.

The paper is outlined as follows. Section 2 addresses the well-posedness of the direct problem. The inverse problem is discussed in section 3. Section 4 presents the numerical examples to demonstrate the effectiveness of the methods. Section 5 concludes the paper with some general remarks and future work. 

\section{The direct problem}

In this section, we study the direct problem and show that it has a unique mild solution. Moreover, an explicit solution is constructed for the direct problem. The explicit solution plays an important role in the inverse problem.

\subsection{Mild solution}

Omitting the spatial variable, we rewrite \eqref{eq:model} into the standard form of an evolution equation
\begin{equation}\label{eq:evolu}
\left\{
\begin{aligned}
du(t)=&~\Delta u(t)dt+q(t)u(t)dB(t),\quad t\in(0,T],\\
u(0)=&~u_0,
\end{aligned}
\right.
\end{equation}
where $\{B(t)\}_{t\ge0}$ is the Brownian motion defined on a complete probability space $(\Omega,\mathcal{F},\mathbb{P})$.

It is known that the operator $-\Delta$ with the homogeneous Dirichlet boundary condition in $D$ admits a non-decreasing sequence of eigenvalues $\{\lambda_k\}_{k=1}^\infty$, which satisfy $0<\lambda_1\leq\lambda_2\leq\cdots$ with $\lambda_k\to\infty$ as $k\to\infty$, and the eigenfunctions of the Dirichlet Laplacian $\{\varphi_k\}_{k=1}^{\infty}$ form an orthonormal basis of $L^2(D)$. Denote by $L_0^2(D)$ the subspace of $L^2(D)$ with the homogeneous boundary condition. Define the interpolation space
\[
\dot H^\alpha(D):={\rm Dom}((-\Delta)^{\frac{\alpha}2})=\left\{f\in L_0^2(D): \sum_{k=1}^\infty\lambda_k^\alpha(f,\varphi_k)_{L^2(D)}^2<\infty\right\},
\]
which is equipped with the norm
\[
\|f\|_\alpha:=\left(\sum_{k=1}^\infty\lambda_k^\alpha(f,\varphi_k)_{L^2(D)}^2\right)^{\frac12}.
\]
If $\alpha$ is a nonnegative integer, it is shown in \cite[Lemma 3.1]{T06} that the norm $\|\cdot\|_\alpha$ is equivalent to the classical Sobolev norm $\|\cdot\|_{H^\alpha(D)}$ under the boundary condition $(-\Delta)^jf=0$ on $\partial D$ for any $j<\frac\alpha2$. It can be verified that $\dot H^0(D)=L^2_0(D)$ and $\dot H^2(D)=H_0^1(D)\cap H^2(D)$.

Let $S(t)=e^{t\Delta}$ be the analytic strongly continuous semigroup generated by the Laplacian on $L^2(D)$. Then it satisfies the following smoothing property (cf. \cite[Lemma 3.2]{T06}).

\begin{lemma}\label{lm:S}
For any $0\le\alpha\le\beta$ and $f\in\dot H^\alpha(D)$, it holds that
\[
\|S(t)f\|_{\beta}\lesssim t^{\frac{\alpha-\beta}2}\|f\|_{\alpha}.
\]
\end{lemma}

Hereinafter, the notation $a\lesssim b$ stands for $a\le Cb$, where $C$ is a positive constant and may differ from line to line, and the notation $\mathbb P$-a.s. indicates that an equation holds almost surely.

The following definition describes the mild solution to a stochastic differential equation (cf. \cite[Definition 6.2.1]{LR15}).

\begin{definition}
An $L^2(D)$-valued adapted process $\{u(t)\}_{t\in[0,T]}$ is called a mild solution to \eqref{eq:evolu} if
\[
u(t)=S(t)u_0+\int_0^tS(t-s)q(s)u(s)dB(s)\quad\mathbb P\text{-a.s.}
\]
for each $t\in[0,T]$ and the stochastic integral is well-defined.
\end{definition}

Using the relation between analytic weak and mild solutions given in \cite[Proposition G.0.5 and Remark G.0.6]{LR15}, we may easily show from \cite[Theorem 4.2.4]{LR15} that \eqref{eq:evolu} has a unique $L^2(D)$-valued mild solution $u$ satisfying
\[
\mathbb E\left[\sup_{t\in[0,T]}\|u(t)\|_{L^2(D)}^2\right]<\infty
\]
provided that $u_0\in L^2(D)$.

\subsection{Well-posedness}

Since the eigenfunctions $\{\varphi_k\}_{k=1}^{\infty}$ of the Dirichlet Laplacian form an orthonormal basis of $L^2(D)$, the solution $u(t)\in L^2(D)$ admits the expansion
\[
u(x,t)=\sum_{k=1}^\infty u_k(t)\varphi_k(x),
\]
where the coefficient $u_k(t):=(u(t),\varphi_k)_{L^2(D)}$ satisfies the stochastic ordinary differential equation
\begin{equation}\label{SODE}
\left\{
\begin{aligned}
du_k(t)=&-\lambda_k u_k(t)dt+q(t)u_k(t)dB(t), \quad  t\in(0,T],\\
u_k(0)=&~u_{0,k}:=(u_0,\varphi_k)_{L^2(D)}.
\end{aligned}
\right.
\end{equation}
By \cite{Evans2013}, the linear equation \eqref{SODE} has a unique solution given explicitly by
\begin{align*}
u_k(t)=u_{0,k} \exp\left(-\lambda_k t+\int_0^t q(s)dB(s)-\frac12\int_0^t q^2(s)ds\right).
\end{align*}
Hence, the solution to (\ref{eq:model}) can be expressed as
\begin{align}\label{eq:solu}
u(x,t)&=\sum_{k=1}^{\infty}u_k(t)\varphi_k(x)\nonumber\\
&=\sum_{k=1}^{\infty}u_{0,k} \exp\left(-\lambda_k t+\int_0^t q(s)dB(s)-\frac12\int_0^t q^2(s)ds\right)\varphi_k(x)\nonumber\\
&=\left[\sum_{k=1}^{\infty}u_{0,k} \exp(-\lambda_k t)\varphi_k(x)\right]\exp\left(\int_0^t q(s)dB(s)-\frac12\int_0^t q^2(s)ds\right)\nonumber\\
&=:v(x,t)Z(t),
\end{align}
where
\[
v(x,t)=\sum_{k=1}^{\infty}u_{0,k} \exp(-\lambda_k t)\varphi_k(x)
\]
and
\[
Z(t)=\exp\left(\int_0^t q(s)dB(s)-\frac12\int_0^t q^2(s)ds\right).
\]
Furthermore, the function $v$ is the unique solution to the initial boundary value problem of the deterministic heat  equation
\begin{equation}\label{eq:v}
\left\{
\begin{aligned}
\frac{\partial v}{\partial t}(x,t)=&~\Delta v(x,t), &&  (x,t)\in{D}\times(0,T],\\
v(x,t)=&~0, && (x,t)\in{\partial D}\times(0,T], \\
v(x,0)=&~u_0(x), && x\in\overline{D}.
\end{aligned}
\right.
\end{equation}

If the initial value $u_0$ is nonnegative, then the positivity of the solution $v$ follows directly from the strong maximum principle for the heat equation (cf. \cite[Theorem 4]{E10}).

\begin{lemma}\label{lm:v}
Let $u_0\in C(D)$  with $u_0\ge0$ being not identically zero. Then the solution $v\in C_1^2(D\times(0,T])\cap C(\overline{D}\times[0,T])$ and satisfies
\[
v(x,t)>0\quad\forall\, (x,t)\in D\times(0,T],
\]
where $C_1^2(D\times(0,T])$ denotes the space of functions belonging to $C^2$ in space and $C^1$ in time.
\end{lemma}

Based on the expression \eqref{eq:solu}, the well-posedness of the problem \eqref{eq:model} can be further obtained in the strong sense. The regularity and strict positivity of the solution can also be deduced under proper assumptions on the initial data.

\begin{theorem}
Let $q\in L^2(0,T)$ and $u_0\in C(D)\cap\dot H^\alpha(D)$ for some $\alpha\ge0$ with $u_0\ge0$ being not identically zero. Then the solution $u\in C([0,T];\dot H^\alpha(D))$ is strictly positive almost surely and satisfies
\[
\mathbb E\left[\sup_{t\in[0,T]}\|u(t)\|_{\alpha}^2\right]\lesssim\|u_0\|_{\alpha}^2.
\]
Moreover, if $\alpha\ge2$, then the mild solution $u$ is also an analytically strong solution to \eqref{eq:model} such that $u\in C([0,T],\dot H^2(D))$ and
\begin{align}\label{eq:strong}
u(t)=u_0+\int_0^t\Delta u(s)ds+\int_0^tu(s)q(s)dB(s)\quad\mathbb P\text{-a.s.}
\end{align}
\end{theorem}

\begin{proof}
For any nonnegative initial data $u_0\in C(D)\cap\dot H^\alpha(D)$, the positivity of $u$ is obtained directly from  \eqref{eq:solu} and the positivity of the solution $v$ to \eqref{eq:v} given in Lemma \ref{lm:v}.

By \eqref{eq:solu}, it holds that
\begin{equation}\label{th1-s1}
\|u(t)\|_{\alpha}^2=\|v(t)\|_{\alpha}^2Z^2(t),
\end{equation}
where we have from Lemma \ref{lm:S} that
\begin{equation}\label{th1-s2}
\|v(t)\|_{\alpha}^2=\|S(t)u_0\|_{\alpha}^2\lesssim\|u_0\|_{\alpha}^2\quad\forall\, t\ge0.
\end{equation}
The stochastic process $\{Z(t)\}_{t\in[0,T]}$ is a continuous martingale (cf. \cite[$(5.2)$ and $(5.17)$]{KS15}) satisfying
\begin{align}\label{eq:z}
\mathbb E[Z(t)]=\mathbb E\left[\exp\left(\int_0^tq(s)dB(s)-\frac12\int_0^tq^2(s)ds\right)\right]=1\quad\forall\,t\ge0.
\end{align}
Using \eqref{th1-s1}--\eqref{th1-s2} and applying the martingale inequality (cf. \cite{Evans2013}) lead to
\begin{align*}
\mathbb E\left[\sup_{t\in[0,T]}\|u(t)\|_{\alpha}^2\right]\lesssim\|u_0\|_{\alpha}^2\mathbb E\left[\sup_{t\in[0,T]}|Z(t)|^2\right]\le4\|u_0\|_{\alpha}^2\mathbb E\left[|Z(T)|^2\right]=4\|u_0\|_{\alpha}^2,
\end{align*}
where the last equality is obtained from the same property as \eqref{eq:z} by using the It\^o formula.

In particular, if $u_0\in \dot H^\alpha(D)$ for some $\alpha\ge2$, then we have $v(t)\in\dot H^2(D)$, $t\in[0,T]$. It follows from  \cite[Proposition G.0.4]{LR15} that $u(t)\in\dot H^2(D)$, $t\in[0,T]$ is also a strong solution satisfying \eqref{eq:strong}.
\end{proof}

\section{The inverse potential problem}

This section is devoted to the inverse potential problem. We present a simple uniqueness result and consider two regularization approaches to overcome the ill-posendess of the inverse problem.

\subsection{Uniqueness}
By \eqref{eq:solu}, together with the strict positivity of $u$ and $v$ in $D$, we get for any fixed $x_*\in D$ that
\begin{align*}
\mathbb E\left[\ln\frac{u(x_*,t)}{v(x_*,t)}\right]=-\frac12\int_0^t q^2(s)ds\quad\forall\,t\in(0,T].
\end{align*}
Taking the derivative of the above equation with respect to $t$ yields
\begin{align}\label{eq:IP}
q^2(t)=-2\frac{d}{dt}\mathbb{E}\left[\ln {\frac{u(x_*,t)}{v(x_*,t)}}\right],
\end{align}
which is the key equation for the inverse problem.

\begin{theorem}\label{iuniq}
Let $q\in L^2(0,T)$ and $u_0\in C(D)\cap\dot H^\alpha(D)$ for some $\alpha\ge0$ with $u_0\ge0$ being not identically zero. Then, for any fixed $x_*\in D$, $\{q^2(t)\}_{t\in(0,T]}$ can be uniquely determined by the data $\{\mathbb{E}\left[\ln u(x_*,t)\right]\}_{t\in(0,T]}$.
\end{theorem}

\begin{proof}
Assume that $u_1$ and $u_2$ are solutions to \eqref{eq:model} corresponding to two potentials $q_1$ and $q_2$, respectively. If $\mathbb{E}[\ln u_1(x_*,t)]=\mathbb{E}[\ln u_2(x_*,t)]$ for any $t\in(0,T]$ and $x_*\in D$, then we have
\begin{align*}
q_1^2(t)-q_2^2(t)&=2\frac{d}{dt}\left(\mathbb{E}\left[\ln {\frac{u_2(x_*,t)}{v(x_*,t)}}\right]-\mathbb{E}\left[\ln {\frac{u_1(x_*,t)}{v(x_*,t)}}\right]\right)\\
&=2\frac{d}{dt}\left(\mathbb{E}\left[\ln u_2(x_*,t)\right]-\mathbb{E}\left[\ln u_1(x_*,t)\right]\right)
=0
\end{align*}
for any $t\in(0,T]$.
\end{proof}

\begin{remark}
 By Theorem \ref{iuniq}, the uniqueness can be obtained for the inverse potential problem if $q$ is nonnegative.
\end{remark}

\subsection{Regularization}\label{sec:regu}

Although  \eqref{eq:IP} provides an explicit formula to reconstruct $q^2$, it is unstable due to the temporal derivative of the data. To handle the instability issue, we consider two regularization methods for the numerical differentiation, the Tikhonov method \cite{QianFuFeng+2006} and the spectral cut-off method \cite{QianFuXW+2006}.

For a given function $\phi(t)$, let $\mathcal F(\phi)(\xi)=\hat{\phi}(\xi)$ and $\phi'(t)$ be the Fourier transform and the first order derivative of $\phi(t)$, respectively. Clearly, we have $\mathcal F(\phi')(\xi)={\rm i}\xi\hat{\phi}(\xi)$, where $\rm i$ is the imaginary unit. Below, we present two simple yet effective regularization methods to implement the numerical differentiation of the noisy data.

\subsubsection{Tikhonov regularization}

The Tikhonov regularization of the time derivative takes the form
\begin{align*}
\mathcal R_1(\phi')(t):=\frac{1}{\sqrt{2\pi}}\int_{-\infty}^{\infty}\frac{{\rm i}\xi}{1+(\mu\xi)^2}\hat{\phi}(\xi)e^{{\rm i}\xi t}d\xi,
\end{align*}
where $\mu>0$ is a regularization parameter. Given $\phi\in H^p(\mathbb R)$ for some $p>1$, it is shown in \cite[Lemma 3.2]{QianFuFeng+2006} that the regularized derivative $\mathcal R_1(\phi')$ satisfies the error estimate
\begin{align*}
\|\phi'-\mathcal R_1(\phi')\|_{L^2(\mathbb R)}&\le\sup_{\xi\in\mathbb R}\left(\frac{\mu^{2}|\xi|^3}{1+(\mu\xi)^2}(1+\xi^2)^{-\frac p2}\right)\|\phi\|_{H^p(\mathbb R)}\\
&\le\max\{\mu^{p-1},\mu^{-1}\}\|\phi\|_{H^p(\mathbb R)}.
\end{align*}

\subsubsection{Spectral cut-off regularization}

The spectral cut-off regularization for the time derivative can be written as
\begin{align*}
\mathcal R_2\left(\phi'\right)(t):=\frac{1}{\sqrt{2\pi}}\int_{-\xi_{\rm max}}^{\xi_{\rm max}}{\rm i}\xi\hat{\phi}(\xi)e^{{\rm i}\xi t}d\xi,
\end{align*}
where $\xi_{\rm max}>0$ is the truncation frequency and plays the role of regularization. Given  $\phi\in H^p(\mathbb R)$ for some $p>1$, it is shown in \cite[Lemma 3.2]{QianFuXW+2006} that the spectral cut-off regularization admits the error estimate
\begin{align*}
\|\phi'-\mathcal R_2(\phi')\|_{L^2(\mathbb R)}&=\left(\int_{|\xi|>\xi_{\rm max}}\frac{\xi^2}{(1+\xi^2)^p}(1+\xi^2)^p|\hat{\phi}(\xi)|^2d\xi\right)^{\frac12}\\
&\le\xi_{\rm max}^{-(p-1)}\|\phi\|_{H^p(\mathbb R)}.
\end{align*}

\section{Numerical experiments}

In this section, we present some numerical experiments for the one-dimensional inverse problem. Several examples are reported to demonstrate the effectiveness of the proposed methods.

\subsection{The synthetic data}

The synthetic data is generated by solving the direct problem numerically and then perturbed by a random noise to test the stability of the methods. We consider the one-dimensional problem with the domain given by $D\times(0,T]=(-a,a)\times(0,T]$, which is discretized into $M$ and $N$ subintervals with uniform step-sizes $h=\frac{2a}M$ and $\tau=\frac TN$ in the spatial and temporal directions, respectively. Denote by
\begin{equation}\label{eq:node}
\begin{aligned}
x_{m}=&~mh,\quad m=0,\pm1,\ldots,\pm M,\\
t_{n}=&~n\tau,\quad \,\,n=0,1,\ldots,N
\end{aligned}
\end{equation}
the discrete points in the spatial and temporal intervals, respectively.

\subsubsection{Numerical approximations}

In \eqref{eq:IP}, it is required to solve the deterministic diffusion equation \eqref{eq:v} and the stochastic diffusion equation \eqref{eq:model} to obtain $v$ and $u$, respectively.

Applying the central difference scheme in space and the backward Euler scheme in time to the deterministic diffusion equation \eqref{eq:v}, we get
\begin{equation}\label{fdm}
\left\{
\begin{aligned}
D_t^+v_m^n=&~D_x^+D_x^-v_m^{n+1},\quad &&m=0,\cdots,\pm(M-1),~n=0,\dots,N-1,\\
v^{n}_{-M}=&~v^{n}_{M}=0, \quad &&n=1,2,\dots,N,\\
v^{0}_{m}=&~u_0(x_{m}), \quad &&m=0,\dots,\pm M,
\end{aligned}
\right.
\end{equation}
where the numerical solution $v_m^n$ is an approximation of $v(x_m,t_n)$ and
\[
D_t^+v_m^n:=\frac{v_m^{n+1}-v_m^n}\tau,\quad D_x^+v_m^n:=\frac{v_{m+1}^n-v_m^n}h,\quad D_x^-v_m^n:=\frac{v_{m}^n-v_{m-1}^n}h
\]
stand for the difference operators. By the classical results (cf. \cite[$(3.77)$]{JS14}), the finite difference scheme \eqref{fdm} is unconditionally stable and admits the error estimate
\[
\max_{1\le n\le N}\|V(t_n)-V^n\|_h\lesssim h^2+\tau,
\]
where $V(t_n)=(v(x_{-M},t_n),\dots,v(x_M,t_n))^\top$, $V^n=(v_{-M}^n,\dots,v_M^n)^\top$, and $\|\cdot\|_h=h^{\frac12}\|\cdot\|_{l^2}$ is the weighted $l^2$-norm for vectors.

For the stochastic diffusion equation \eqref{eq:model}, employing the central difference scheme in space and the implicit Euler--Maruyama scheme in time, we have
\begin{equation}\label{eq:num}
\left\{
\begin{aligned}
D_t^+u^{n}_{m}=&~D_x^+D_x^-u^{n+1}_{m}+q(t_{n})u^{n}_{m}D_t^+B(t_{n}),\quad &&m=0,\dots,\pm(M-1),~n=0,\dots,N-1,\\
u^{n}_{-M}=&~u^{n}_{M}=0, \quad &&n=1,2,\dots,N,\\
u^{0}_{m}=&~u_0(x_{m}),\quad &&m=0,\dots,\pm M,
\end{aligned}
\right.
\end{equation}
where
\[
D_t^+B(t_n)=\frac{B(t_{n+1})-B(t_n)}\tau\sim\sqrt{\tau}\eta_{n+1},\quad n=0,\cdots,N-1
\]
with $\{\eta_n\}_{n=1,\cdots,N}$ being independent and identically distributed standard normal random variables. The notation $a\sim b$ means that $a$ and $b$ have the same distribution. We mention that the stochastic It\^o integral should be approximated by evaluating the integrant at the left endpoint $t_n$ on each subinterval $[t_n,t_{n+1}]$. If the initial value is smooth enough, e.g., $u_0\in C^3(D)$, the numerical scheme \eqref{eq:num} is also unconditionally stable and has the error estimate
\[
\left(\mathbb E\|U(t_n)-U^n\|_h^2\right)^{\frac12}\lesssim h^2+\tau^{\frac12},
\]
where the notations $U(t_n)$ and $U^n$ are defined similarly as $V(t_n)$ and $V_n$. We refer to \cite[Theorem 3.1 ${\rm (iii)}$]{G99} for the detailed convergence analysis of the scheme \eqref{eq:num}.

\subsubsection{Noisy data and regularization}

For any fixed interior point $x_*\in D$, without loss of generality, we may assume that $x_*=x_m$ is one of the grid points in \eqref{eq:node}. Then the data $\{u(x_m,t)\}_{t\in(0,T]}$ can be generated as $\{u_m^n\}_{n=0,\dots,N}$ by implementing the numerical scheme \eqref{eq:num}. In addition, we add some random noise to the data in order to test the stability of the reconstructions. Let $u^{n,\epsilon}_m$ be the noisy data given by
\begin{align*}
u^{n,\epsilon}_m=u_m^n\left(1+\epsilon\zeta_n\right),\quad n=0,\dots,N,
\end{align*}
where $\epsilon>0$ is the noise level and $\{\zeta_n\}_{n=0,\cdots,N}$ are independent standard normal random variables with mean zero and variance one.

Once $v_m^n, u^{n, \epsilon}_m, n=0,\dots, N$ are available, we may define the data points
\begin{equation}\label{psin}
\psi^n:=\mathbb{E}\left[\ln {\frac{u^{n,\epsilon}_m}{v_m^n}}\right]=\mathbb E[\ln u^{n,\epsilon}_m]-\ln v_m^n,\quad n=0,\dots,N
\end{equation}
and the linearly interpolated data function
\begin{align}\label{psifun}
\psi(t)=\frac{t_{n+1}-t}{\tau}\psi^n+\frac{t-t_n}{\tau}\psi^{n+1},\quad t\in[t_n,t_{n+1}].
\end{align}

The function $\psi$ needs to be periodically extended such that the fast Fourier transform (FFT) can be applied to compute its derivative efficiently. For example, the domain of $\psi$ can be extended from $[0,T]$ to $[-T,2T]$. We adopt the cubic smoothing spline developed in \cite{Elden2000} and denote the extension by $\Psi$, which satisfies $\Psi(-T)=\Psi(2T)=0$. Specifically, the extension to the domain $[T,2T]$ is generated by using a cubic smoothing spline based on the last 16 components in $\{\psi^n\}_{n=0,\cdots,N}$ and assigning the last 7 components in $\{\psi^n\}_{n=N+1,\cdots,2N}$ to be zeros. Then the extended data $\{\psi^n\}_{n=N+1,\cdots,2N}$ are obtained by the interpolation. The extension to the domain $[-T,0]$ can be constructed similarly.

Based on \eqref{eq:IP} and \eqref{psifun}, the function $q^2$ can be approximated by $-2\psi'$, where the derivative $\psi'$ is computed numerically by using the regularization methods introduced in section \ref{sec:regu} together with the FFT. 

\subsection{Numerical examples}

In this section, we present three numerical examples to illustrate the performance of the reconstruction for potentials with different regularity.

In all the experiments, we take the following setup: the computational domain is $[-a,a]\times[0,T]$ with $a=1$ and $T=1$, the initial condition of \eqref{eq:model} is chosen as $u_0(x)=e^{-16x^2}$, the observation point $x_*=x_0=0$,  and the numbers of subintervals $M=50$ and $N=2^7$. In practice, the expectation in \eqref{eq:IP} is approximated by the average of $P$ realizations, where the choice of $P$ will be specified in the following examples.

\subsubsection{Example 1}

The exact potential function is $q(t)=\sin(\pi t)$, $t\in[0, 1]$. Using this potential function as a representative example, we examine the influence of various parameters on the reconstructions and present the corresponding numerical results.

\begin{figure}
  \includegraphics[width=0.32\textwidth]{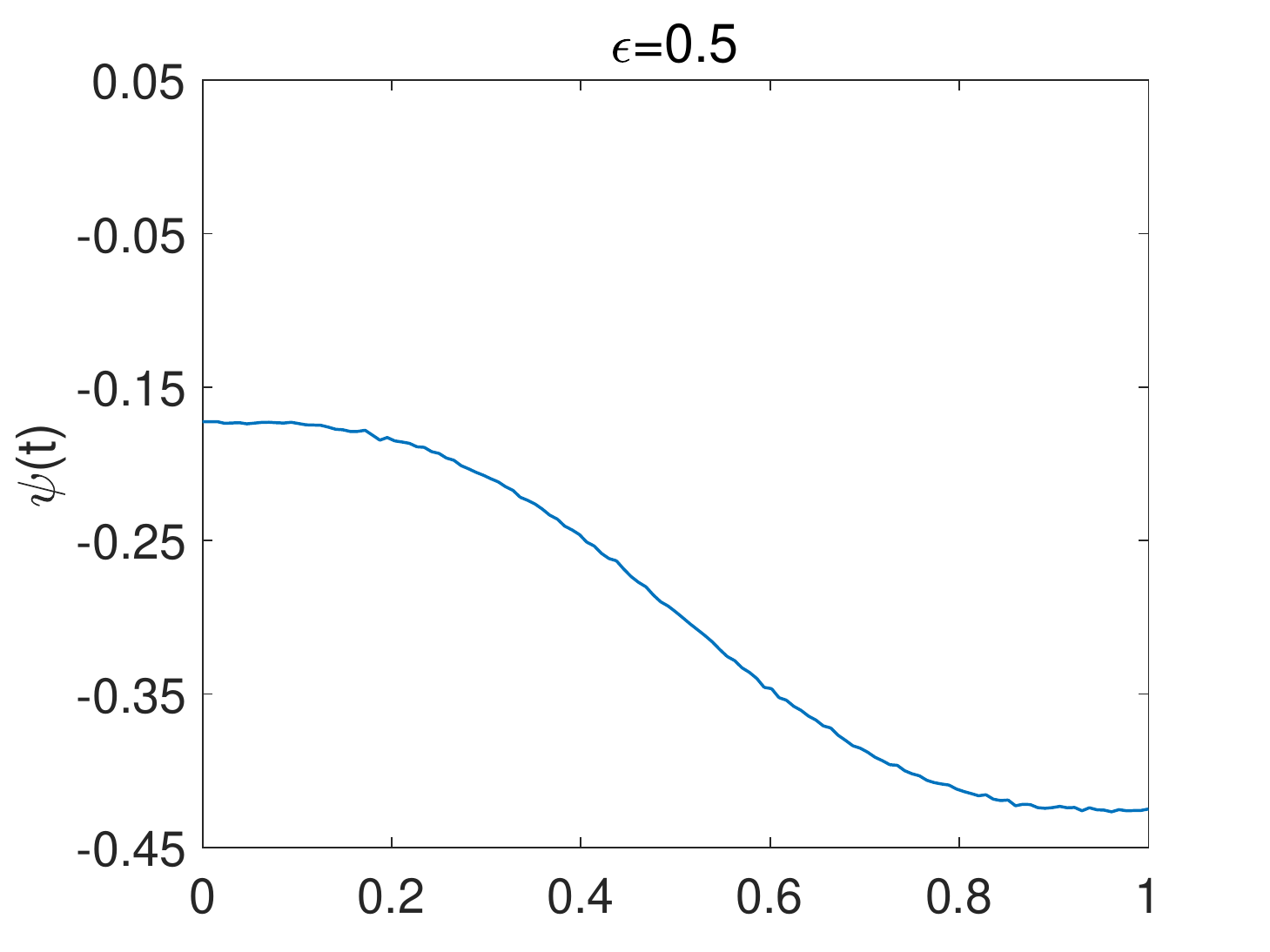}
  \includegraphics[width=0.32\textwidth]{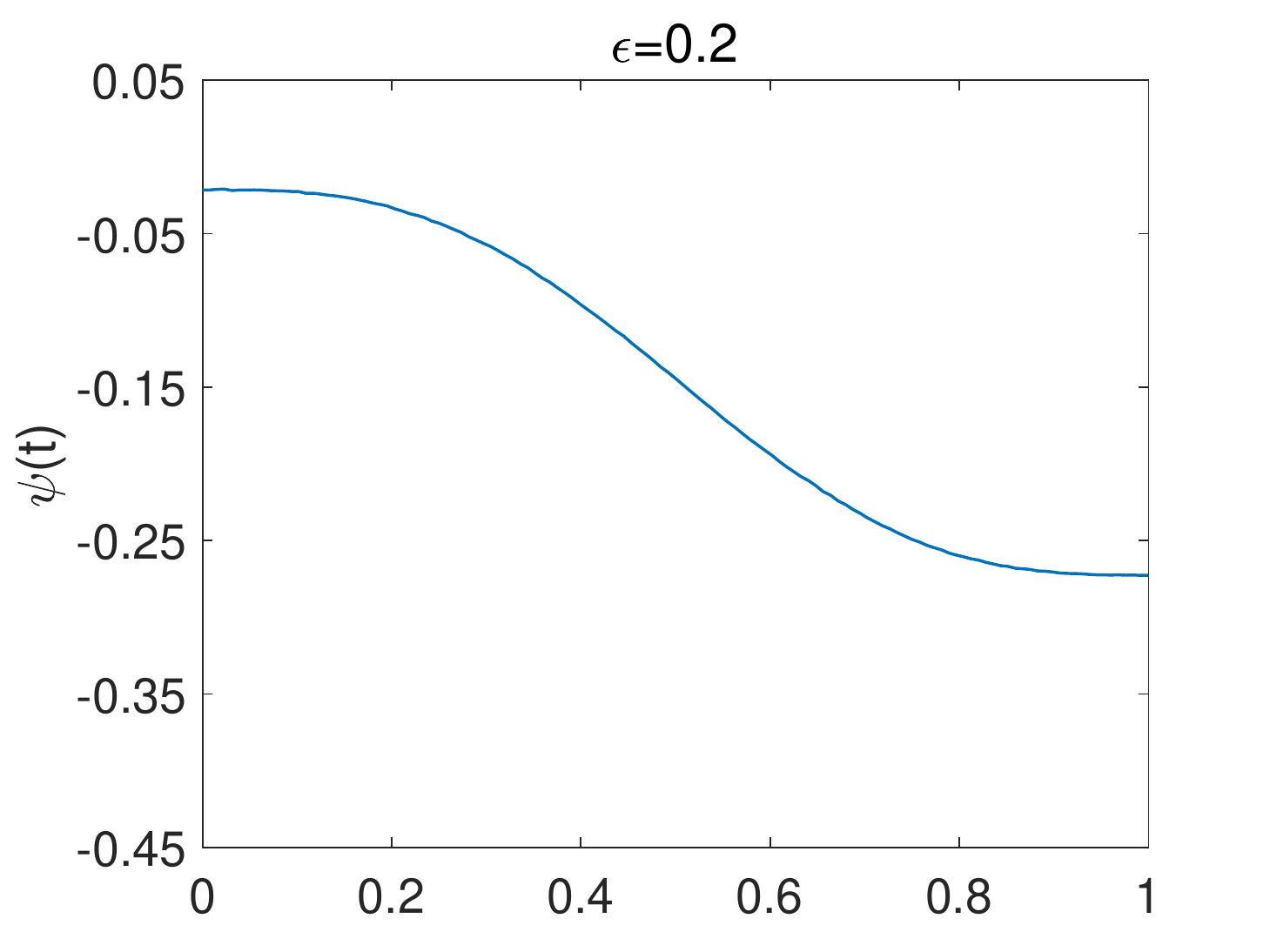}
  \includegraphics[width=0.32\textwidth]{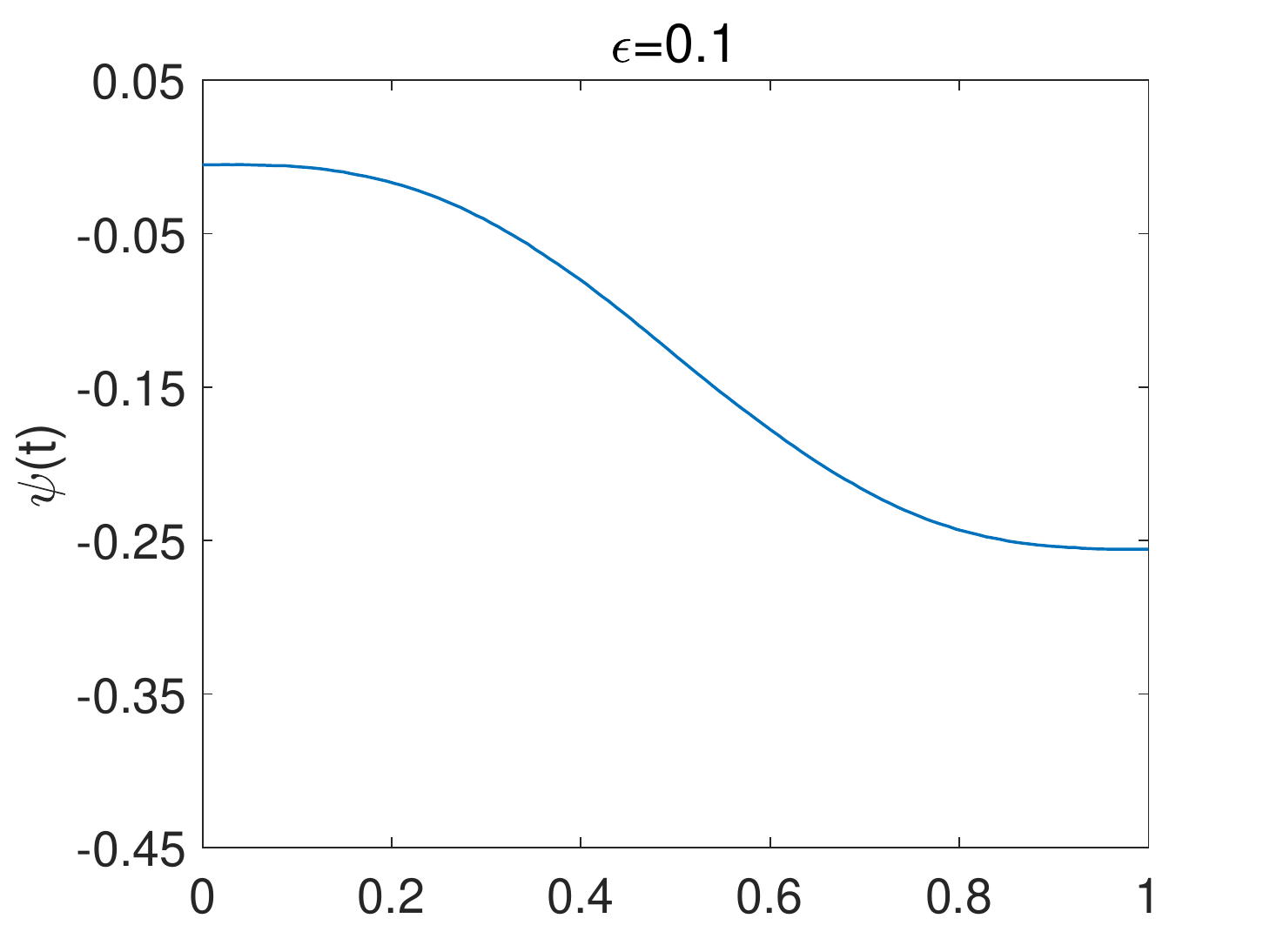}\\
  \includegraphics[width=0.32\textwidth]{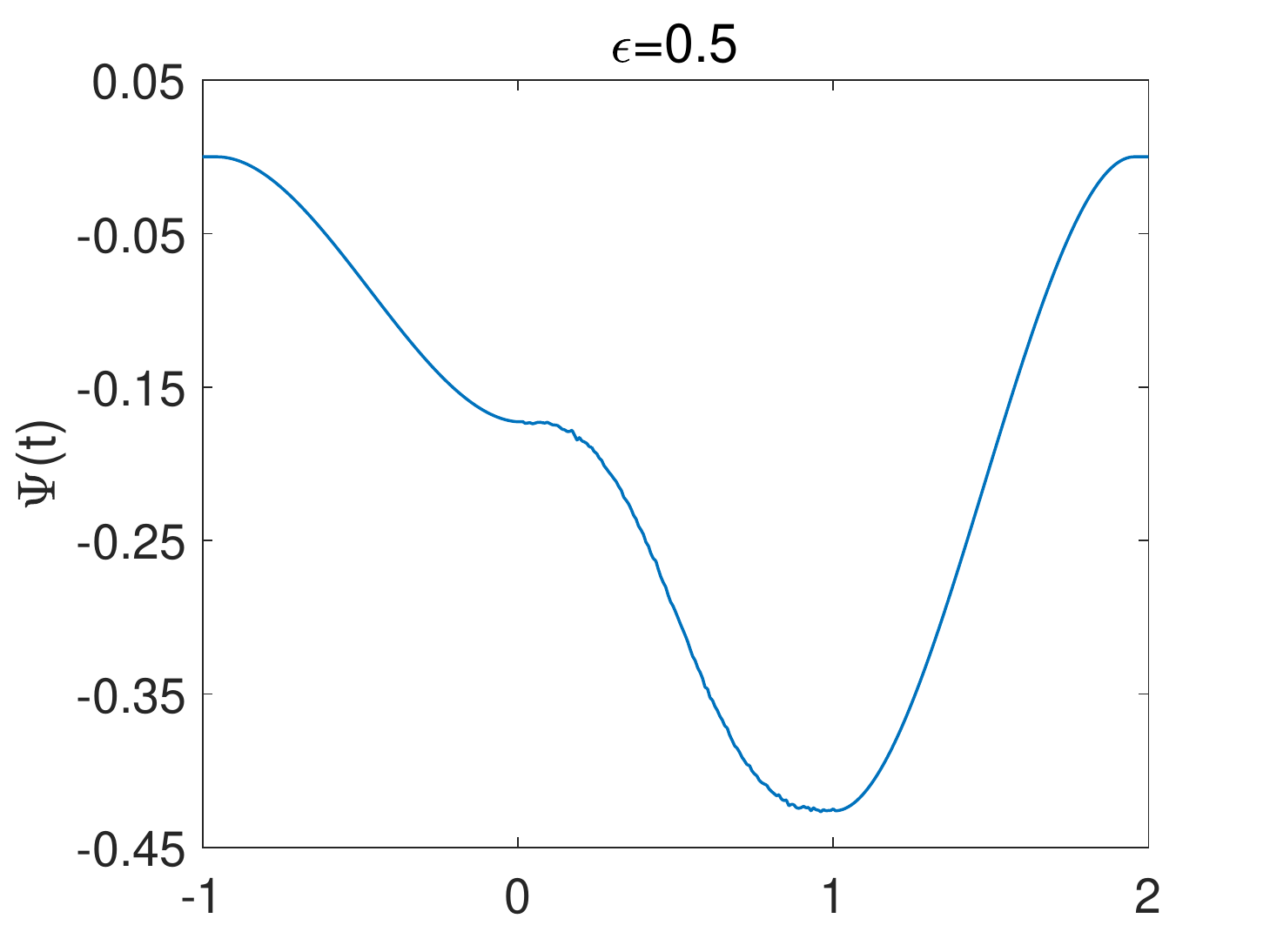}
  \includegraphics[width=0.32\textwidth]{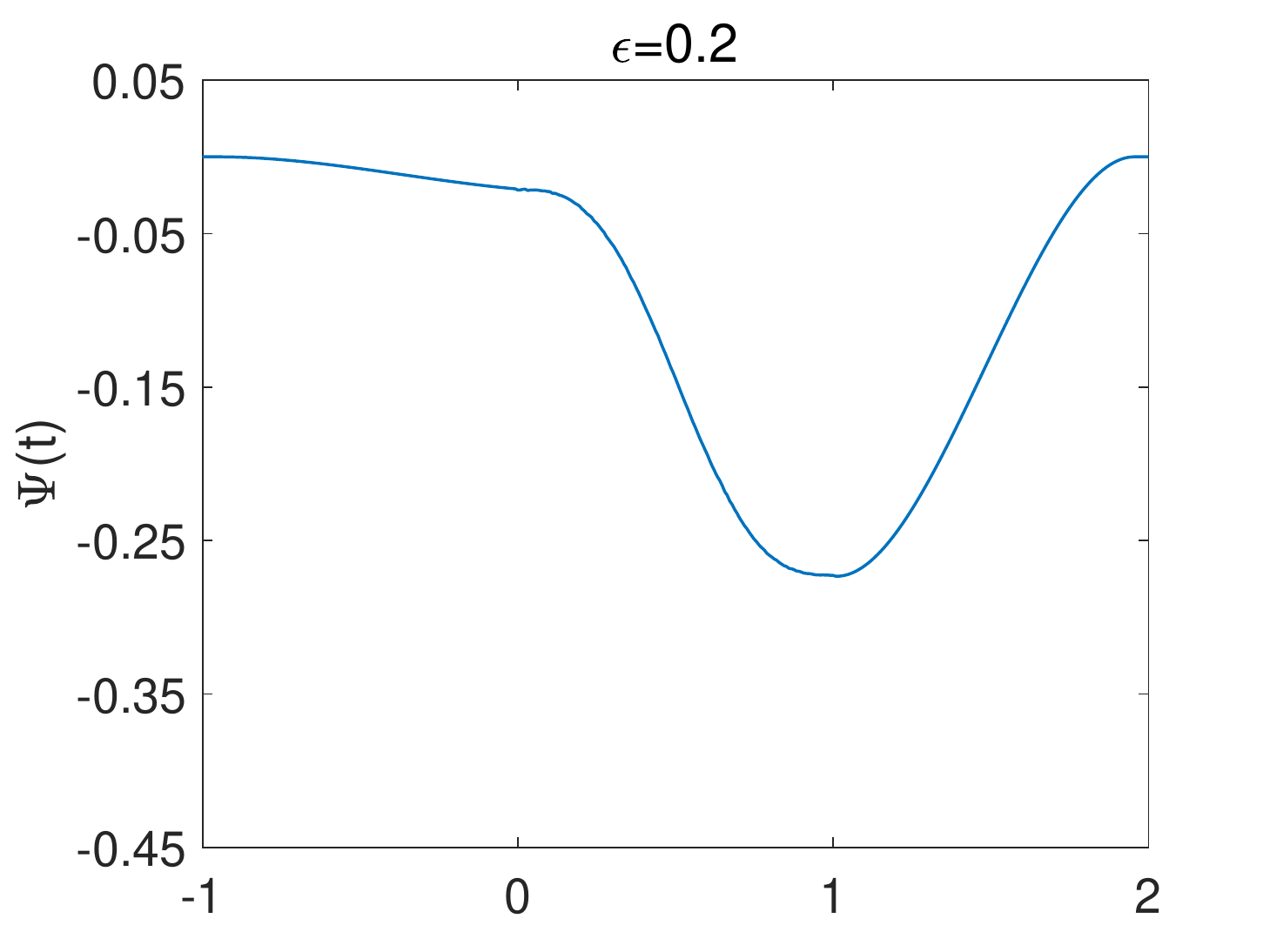}
  \includegraphics[width=0.32\textwidth]{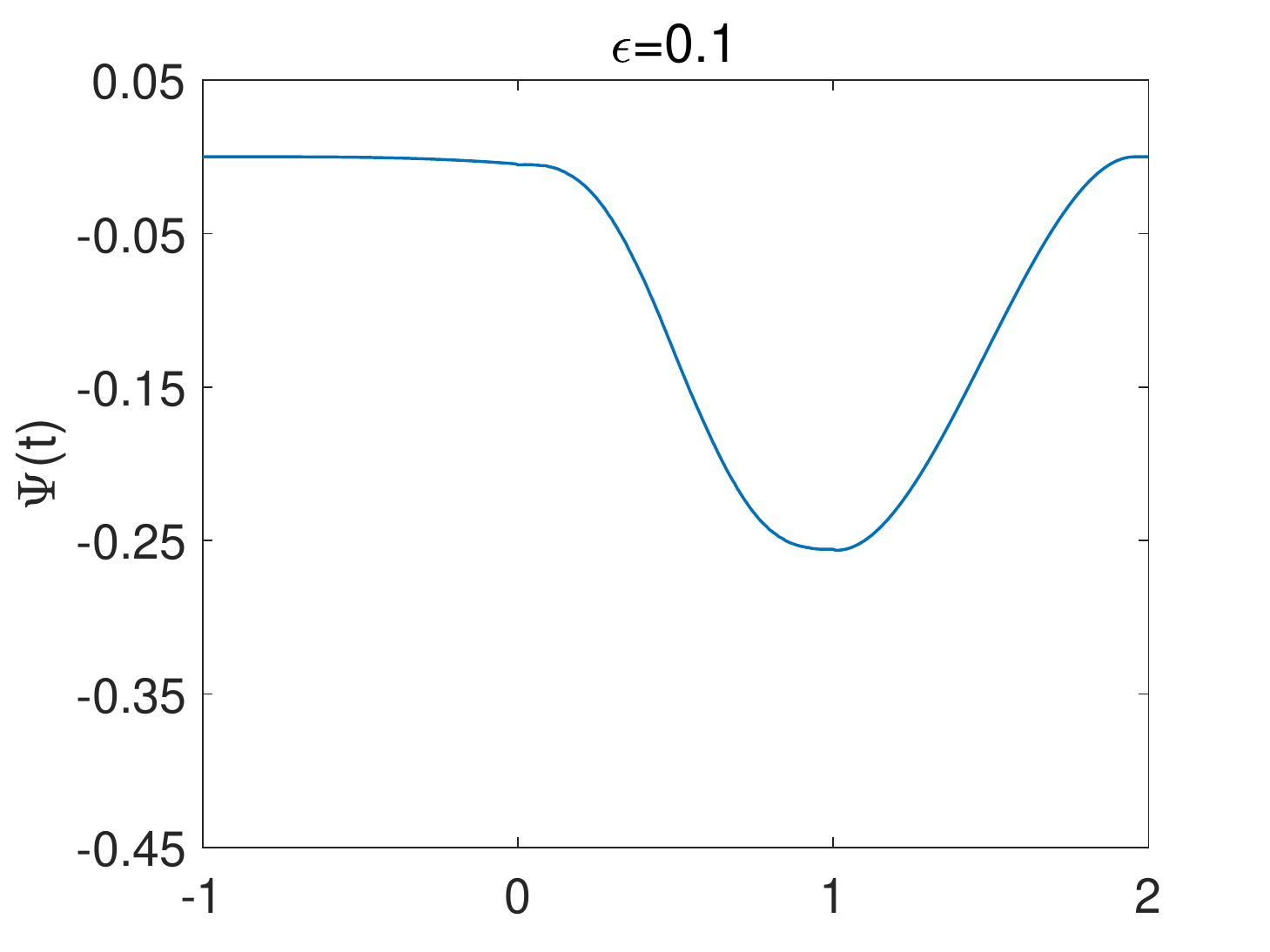}
  \caption{Example 1: the data function $\psi(t)$ on $[0,1]$ (top row) and the corresponding periodization $\Psi(t)$ on $[-1,2]$ (bottom row) at different noise levels ($\epsilon=0.5, 0.2, 0.1$) with a fixed number of realizations ($P=10^6$).}\label{e1psip6}
\end{figure}

\begin{figure}
  \includegraphics[width=0.32\textwidth]{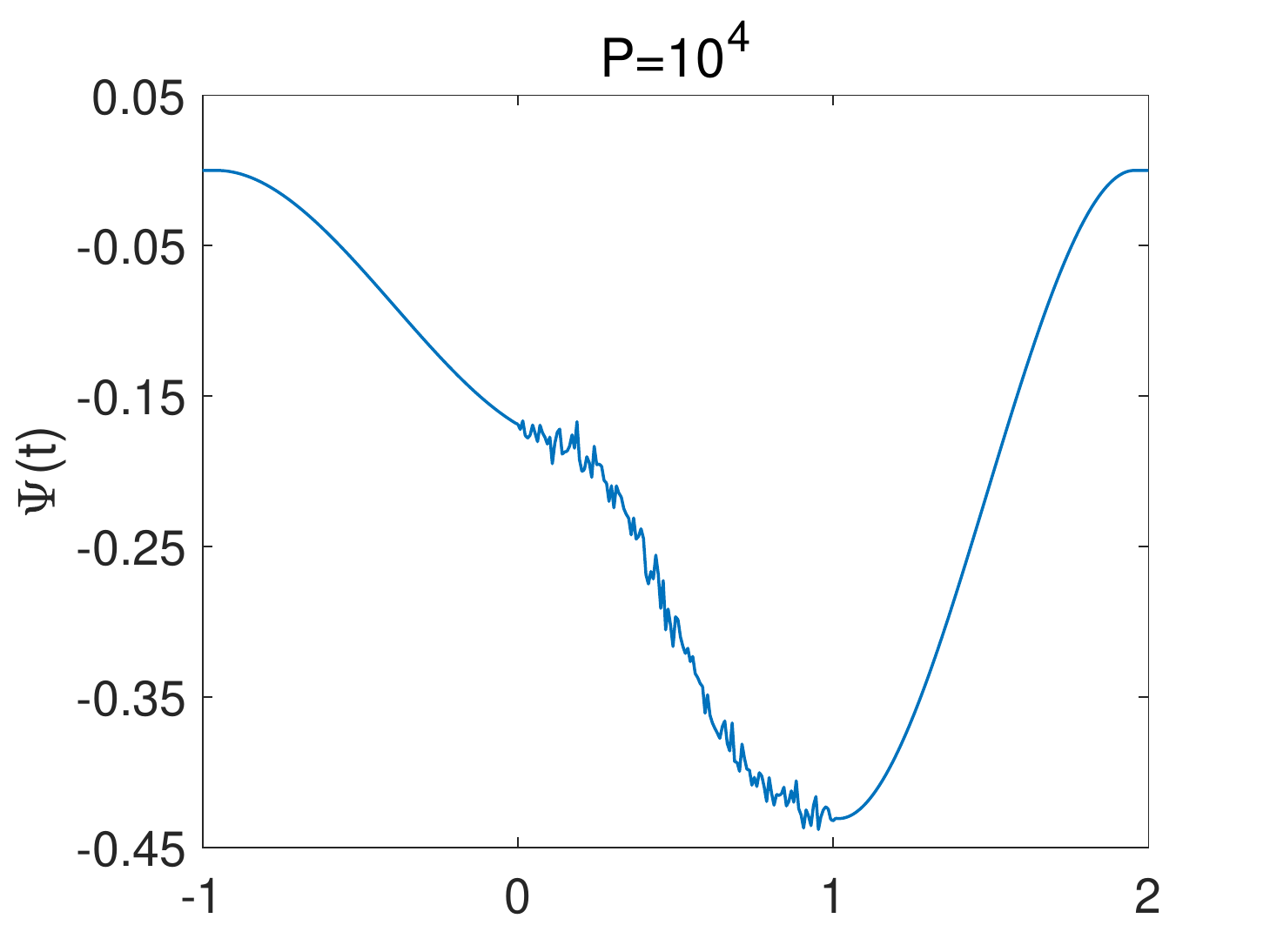}
  \includegraphics[width=0.32\textwidth]{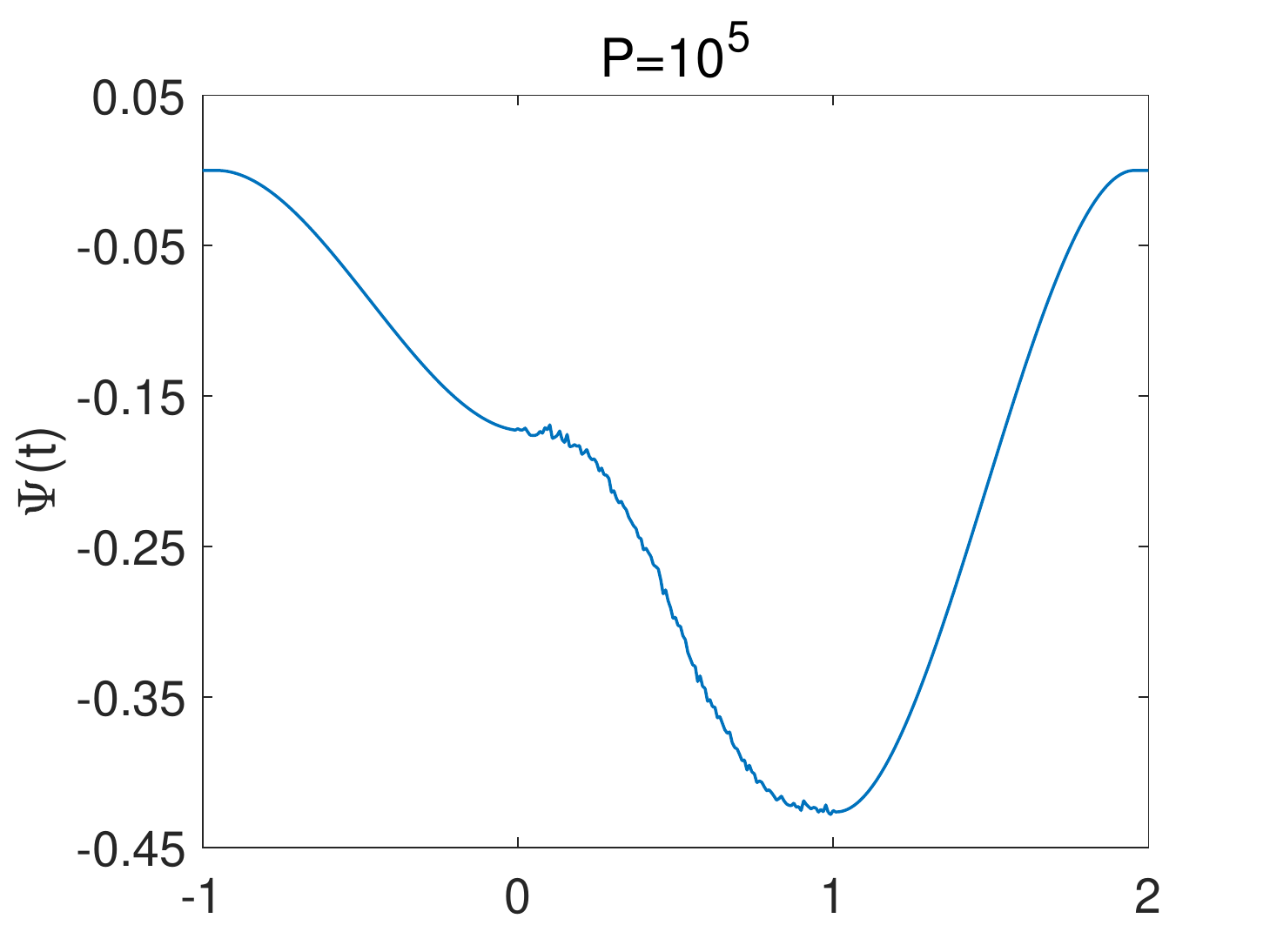}
  \includegraphics[width=0.32\textwidth]{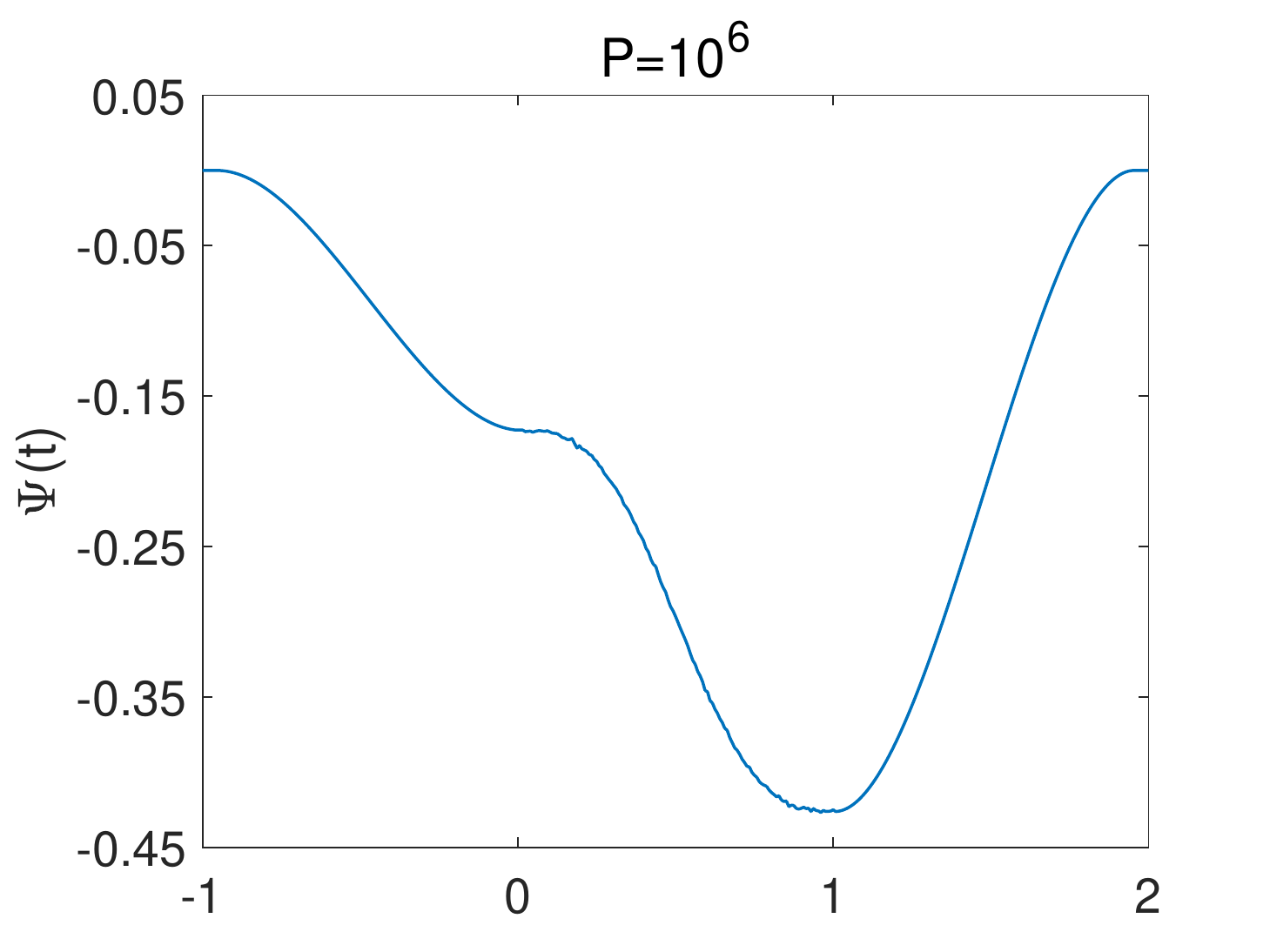}
  \caption{Example 1: the periodized data function $\Psi(t)$ on $[-1,2]$ with a different number of realizations ($P=10^4,10^5,10^6$) at a fixed noise level ($\epsilon=0.5$).}\label{e1psip54}
\end{figure}

\begin{figure}
  \includegraphics[width=0.32\textwidth]{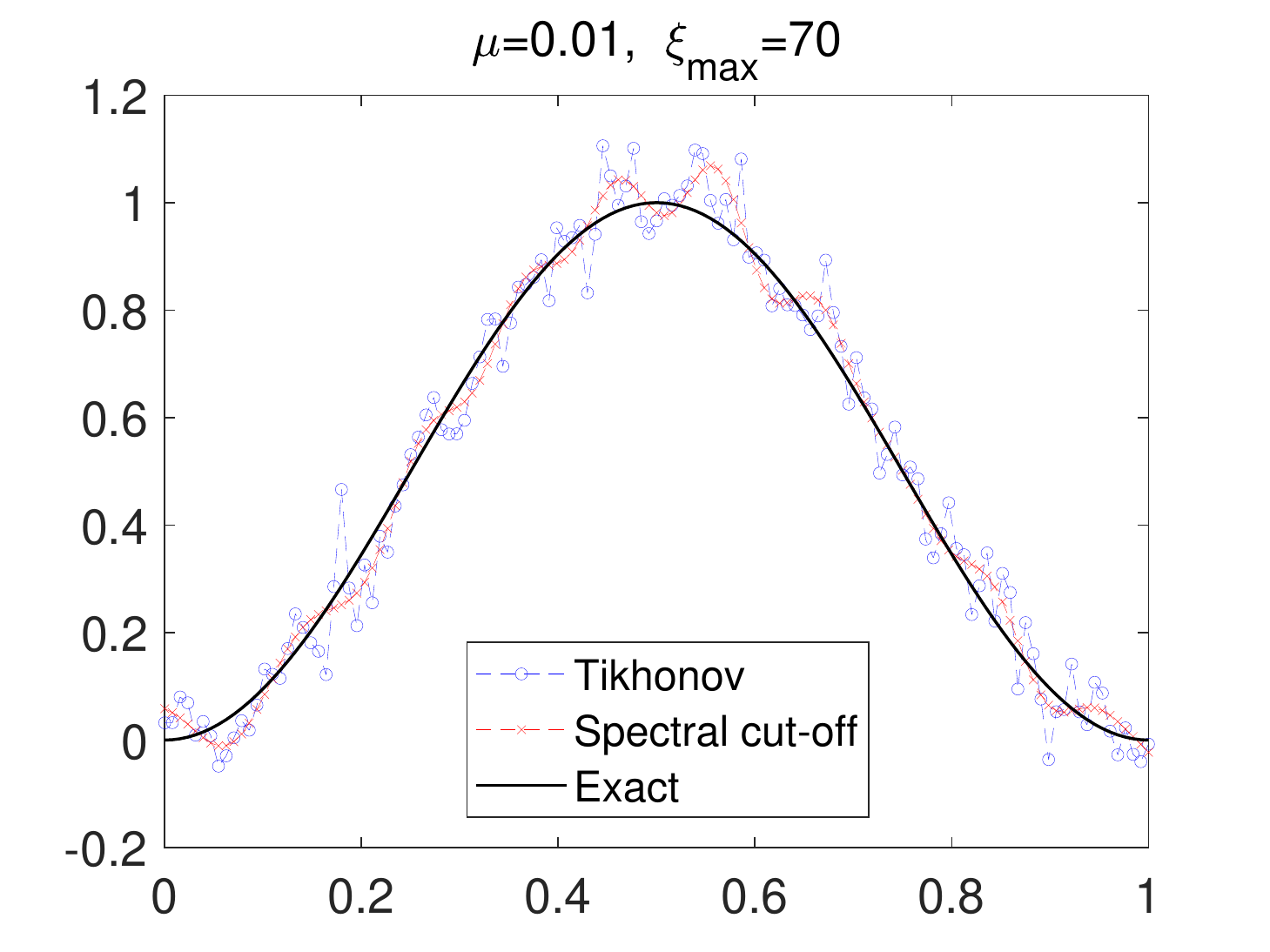}
  \includegraphics[width=0.32\textwidth]{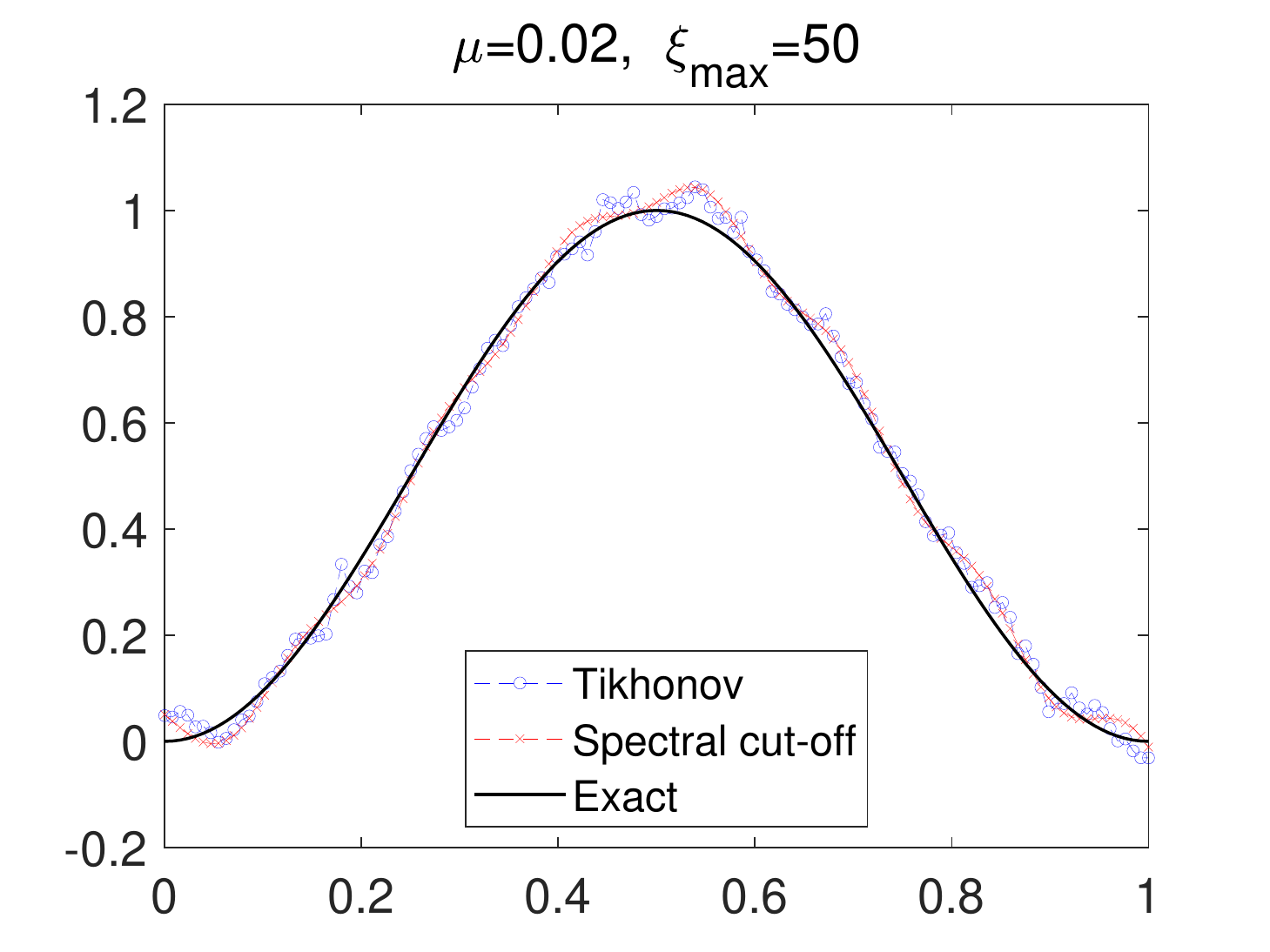}
  \includegraphics[width=0.32\textwidth]{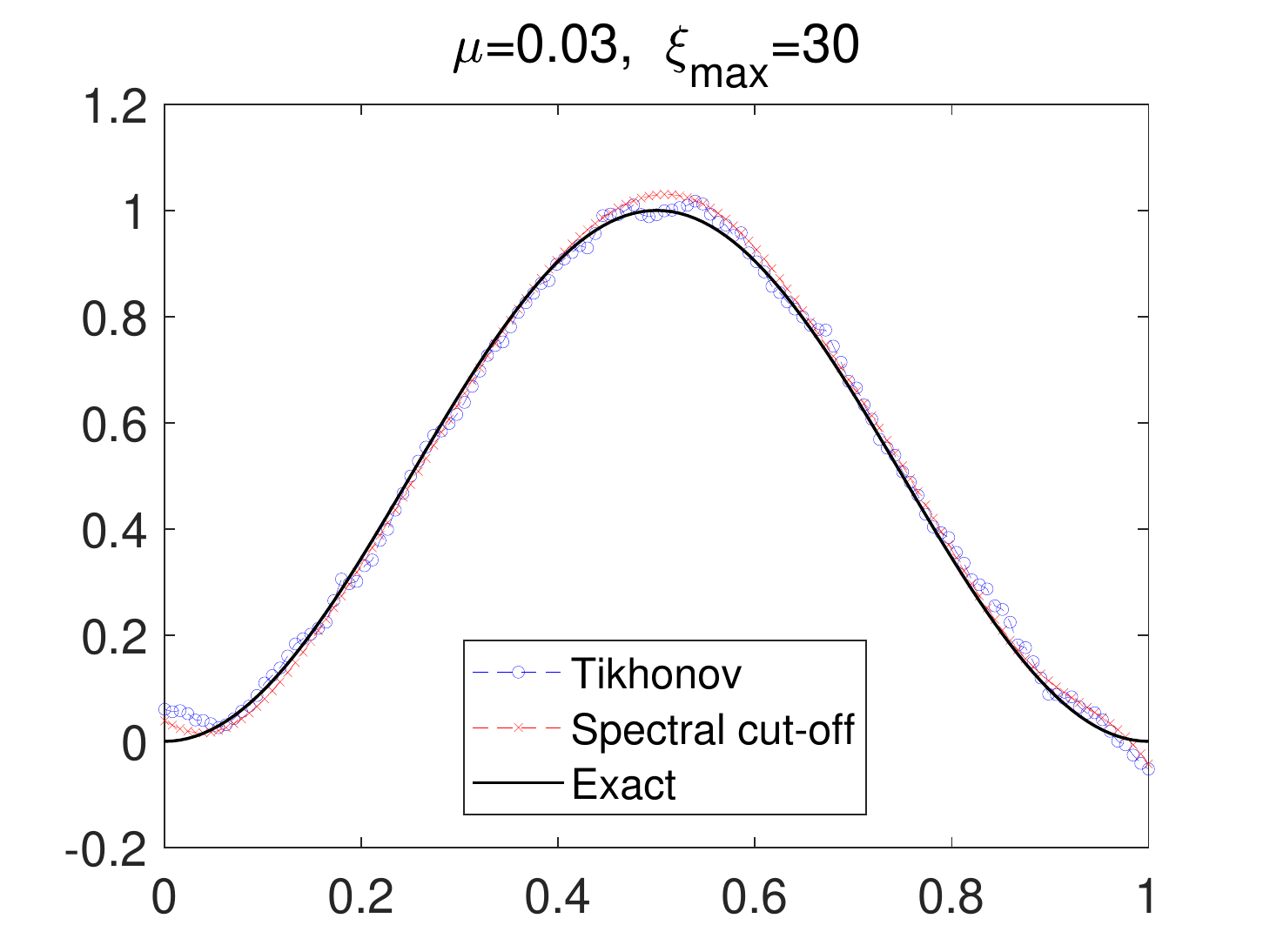}\\
  \includegraphics[width=0.32\textwidth]{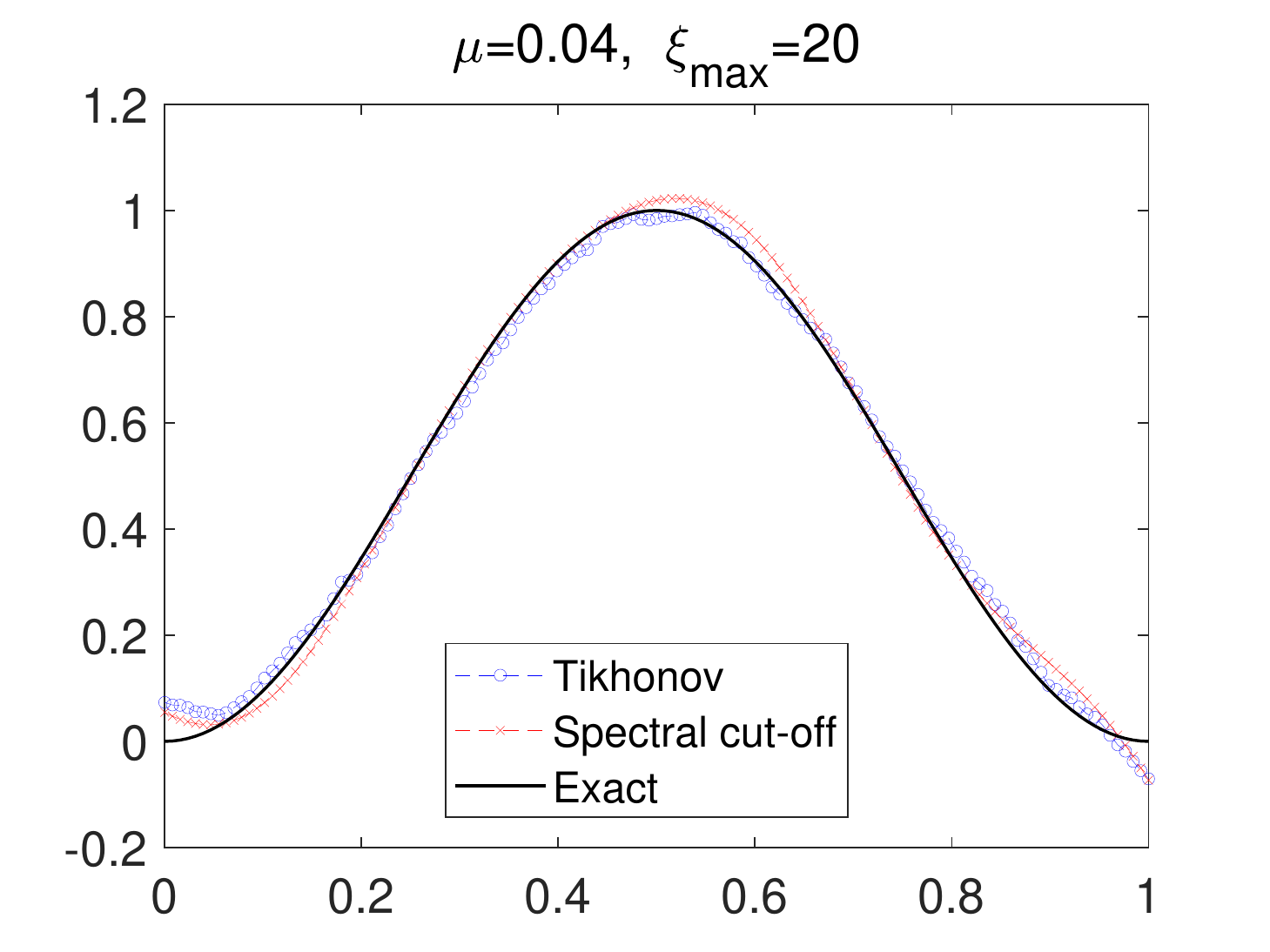}
  \includegraphics[width=0.32\textwidth]{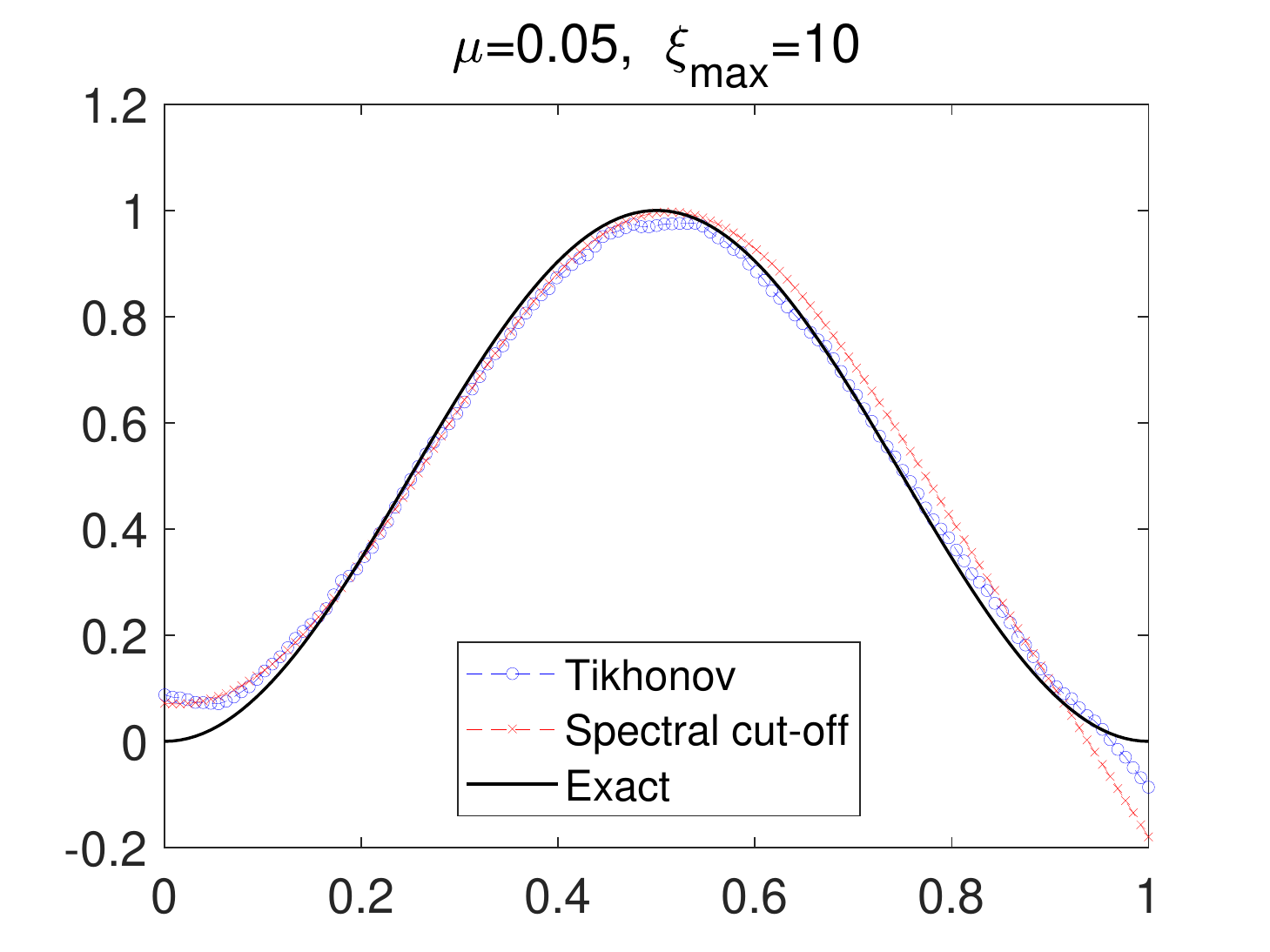}
  \includegraphics[width=0.32\textwidth]{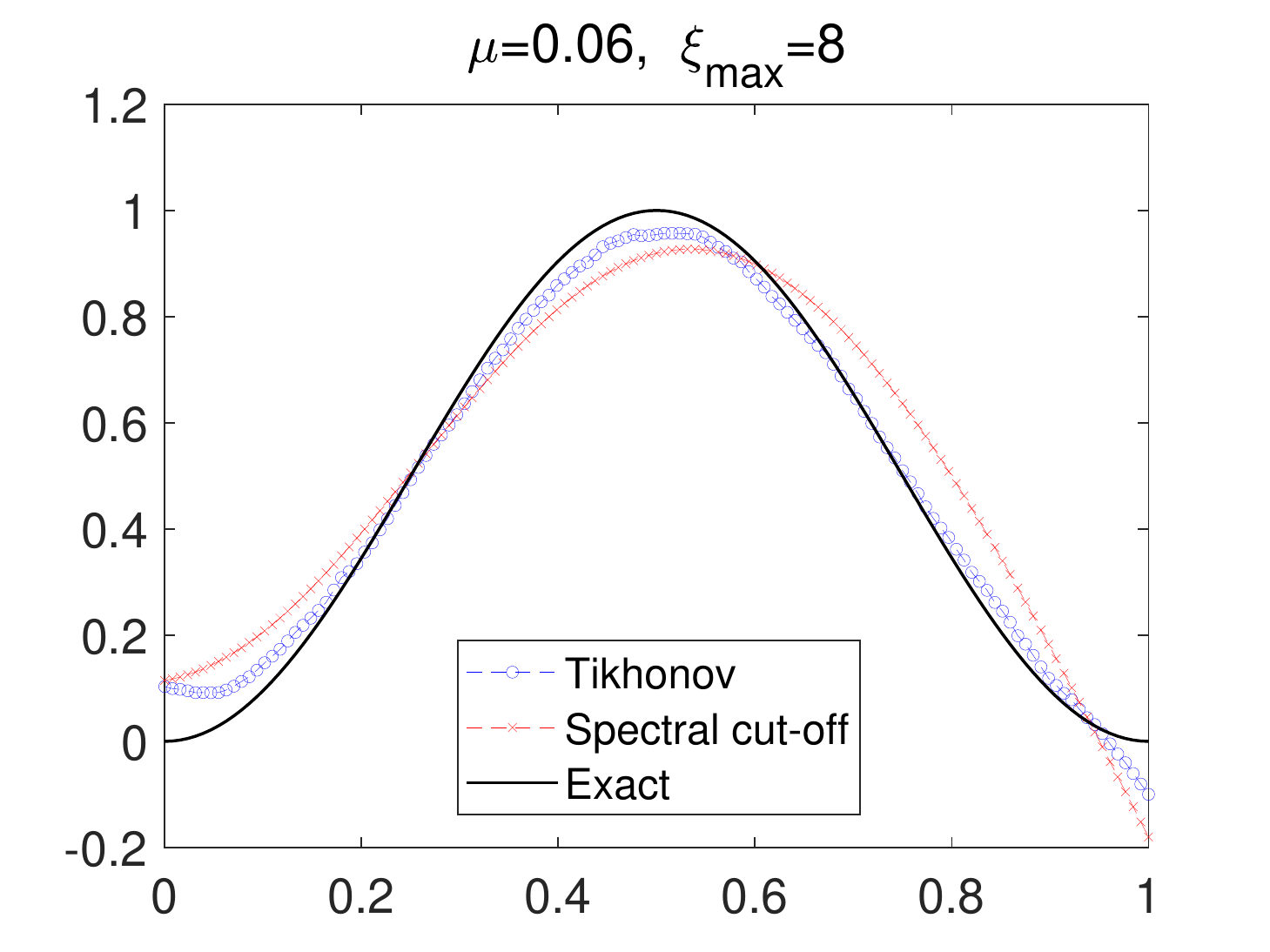}
  \caption{Example 1: the reconstruction of $q^2$ with different regularization parameters at a fixed noise level ($\epsilon=0.5$) and a fixed number of realizations ($P=10^6$).}\label{e105p6}
\end{figure}

\begin{figure}
  \includegraphics[width=0.32\textwidth]{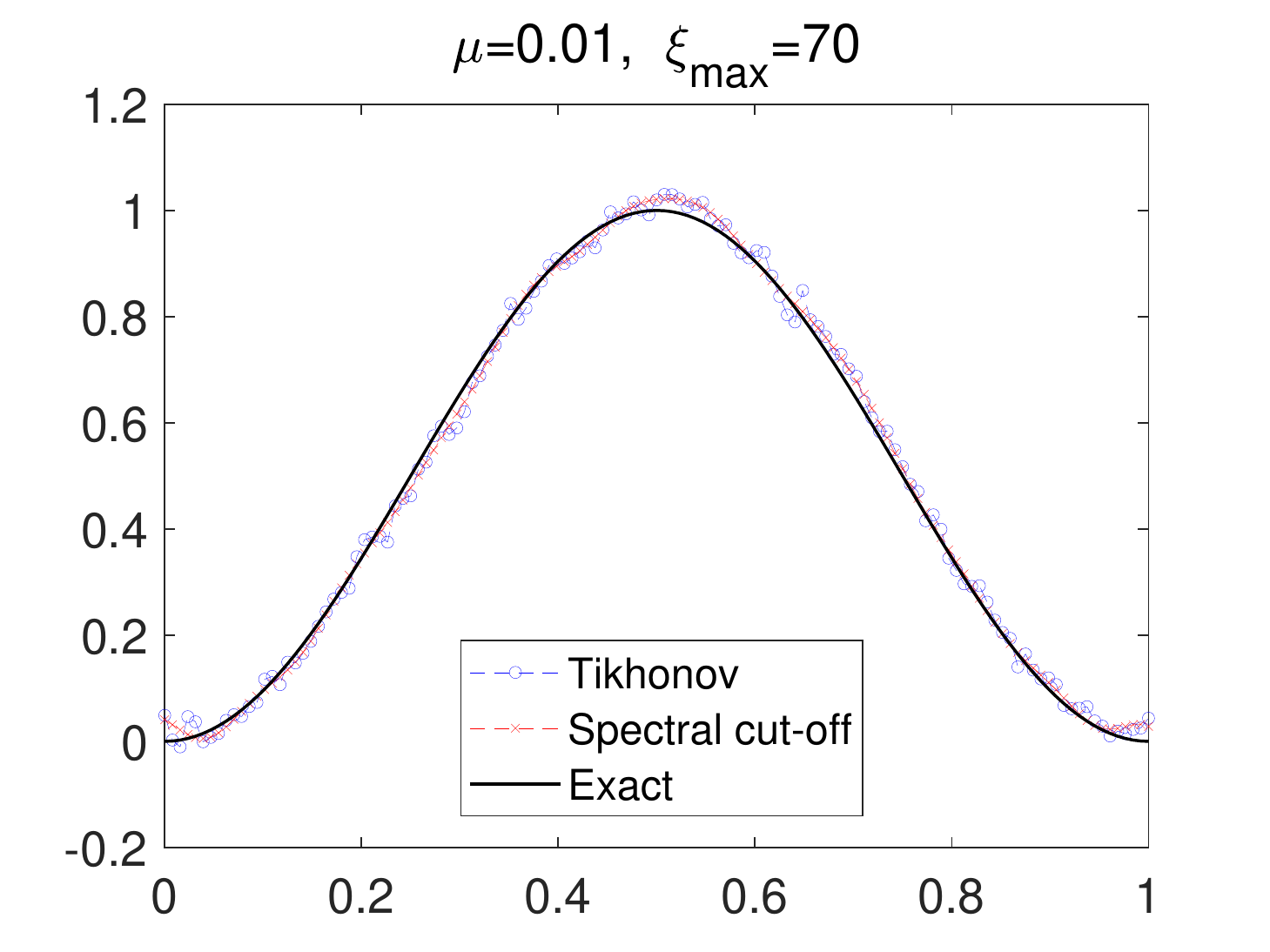}
  \includegraphics[width=0.32\textwidth]{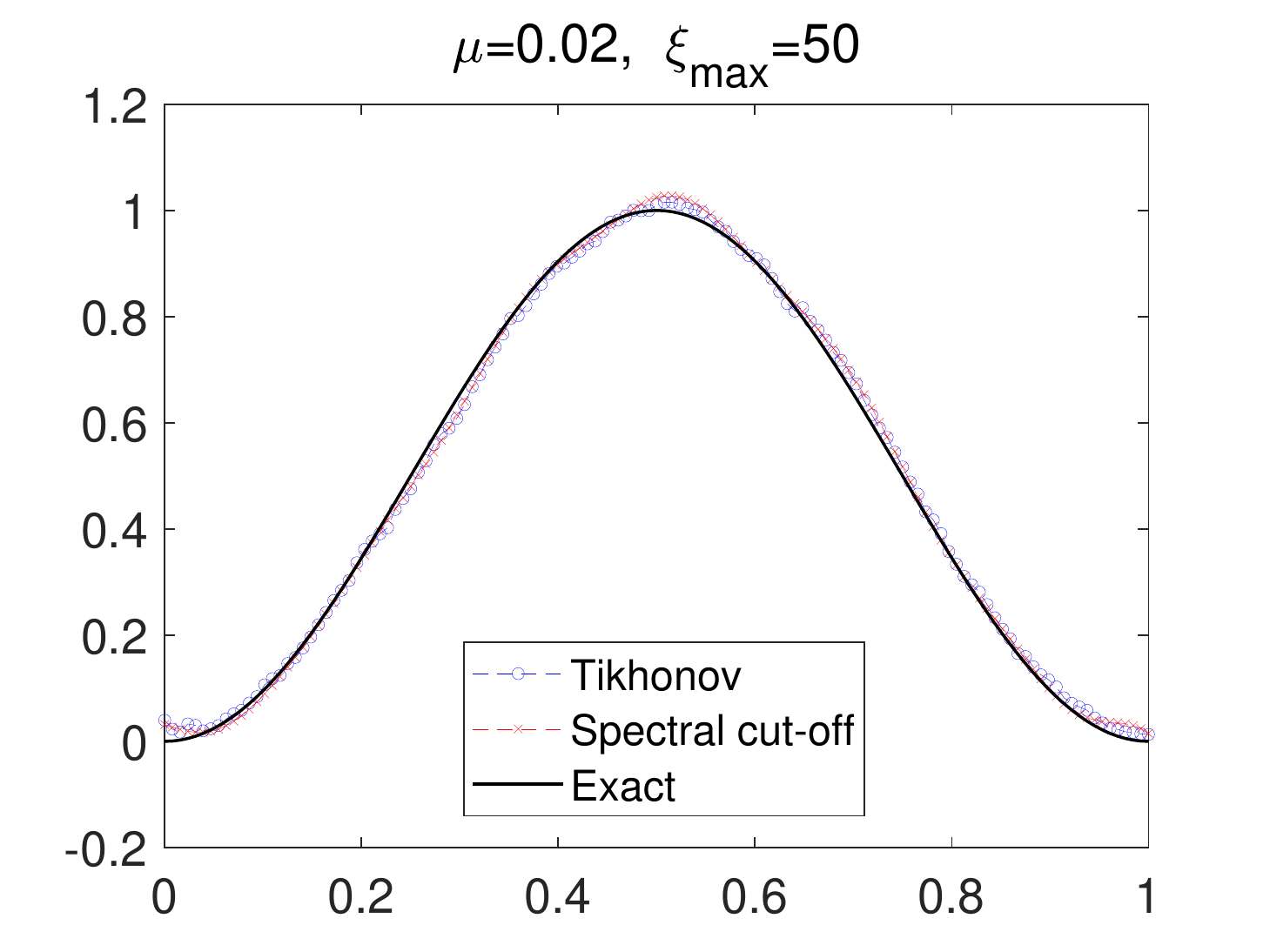}
  \includegraphics[width=0.32\textwidth]{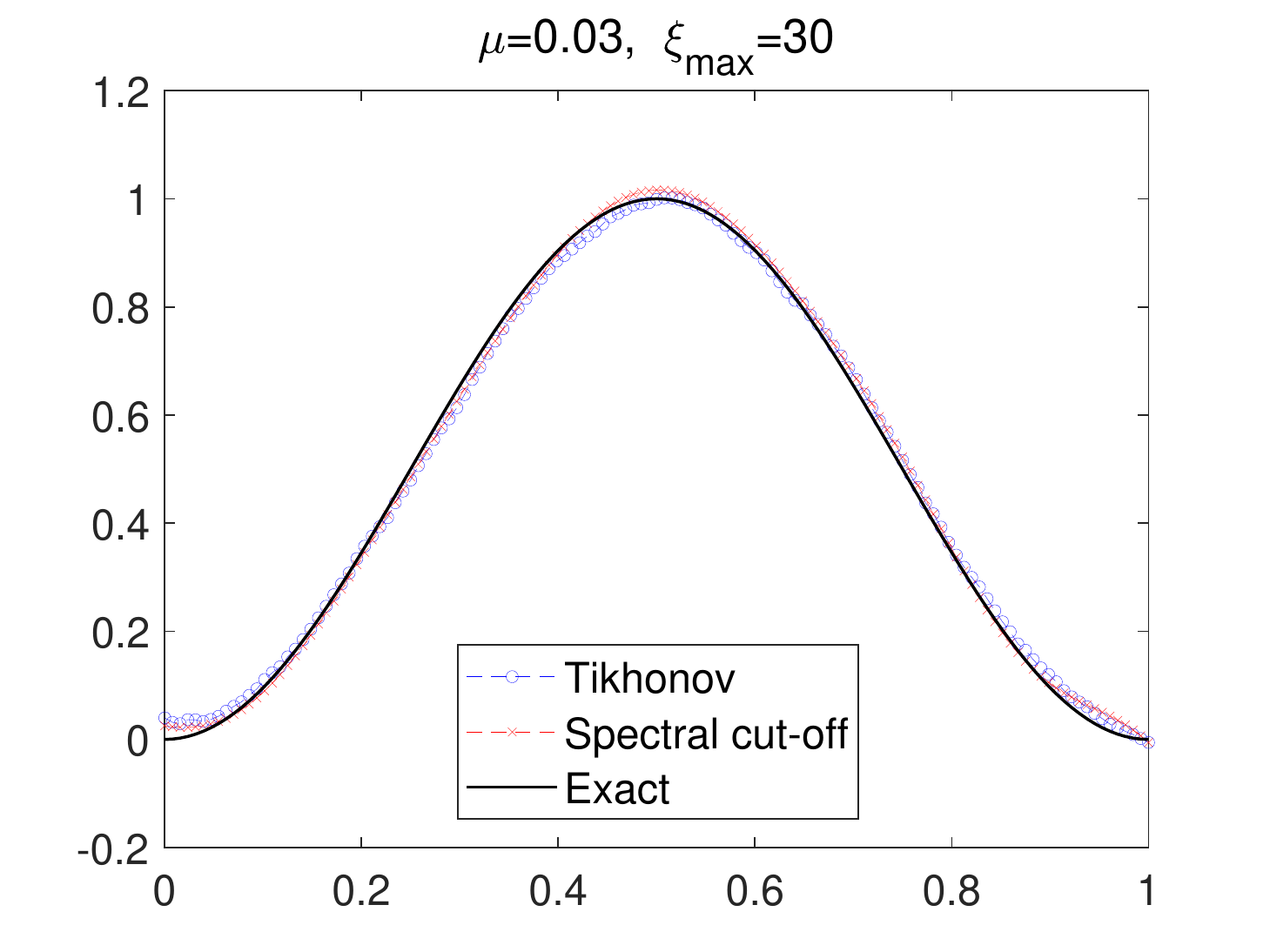}\\
  \includegraphics[width=0.32\textwidth]{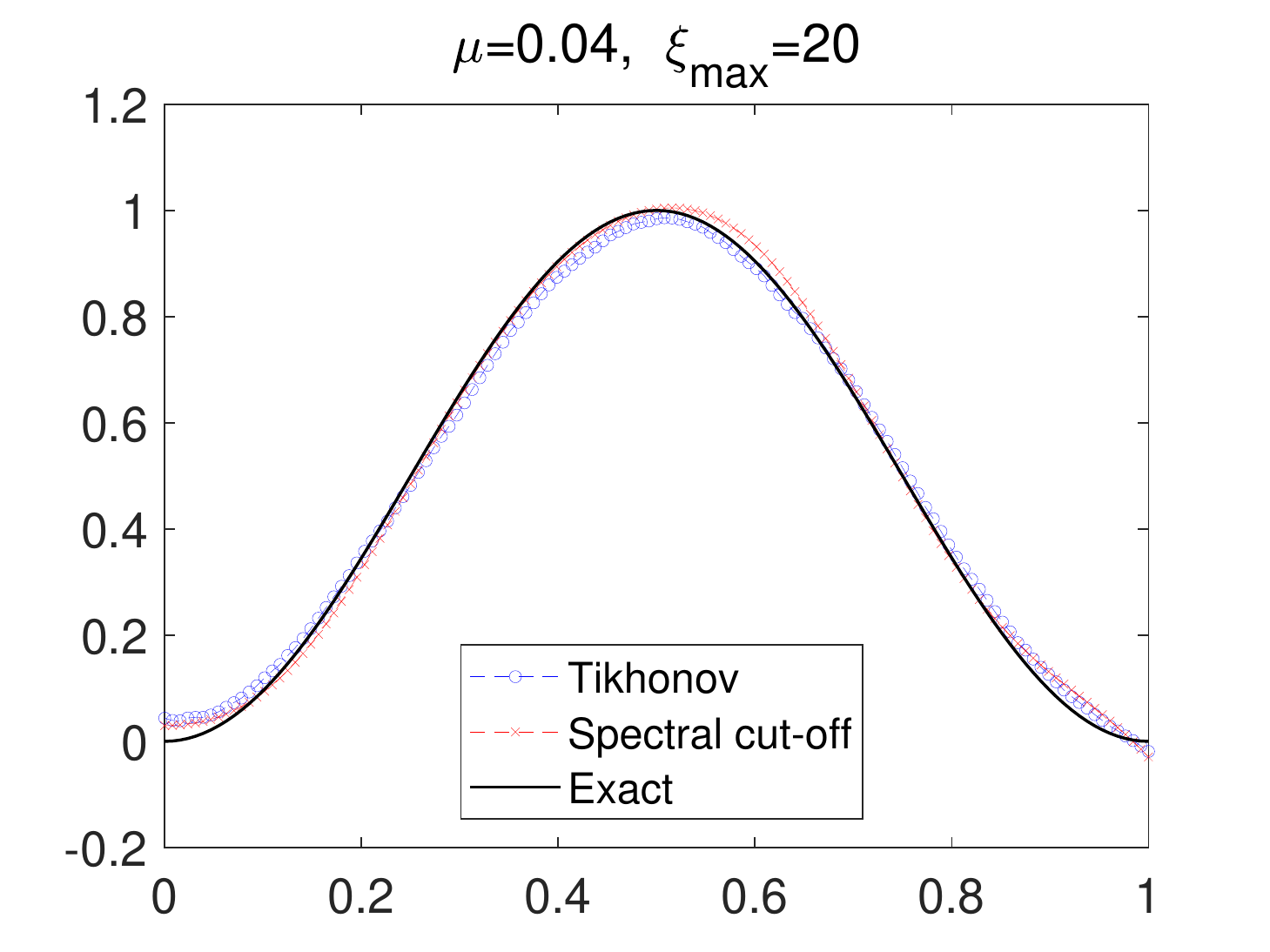}
  \includegraphics[width=0.32\textwidth]{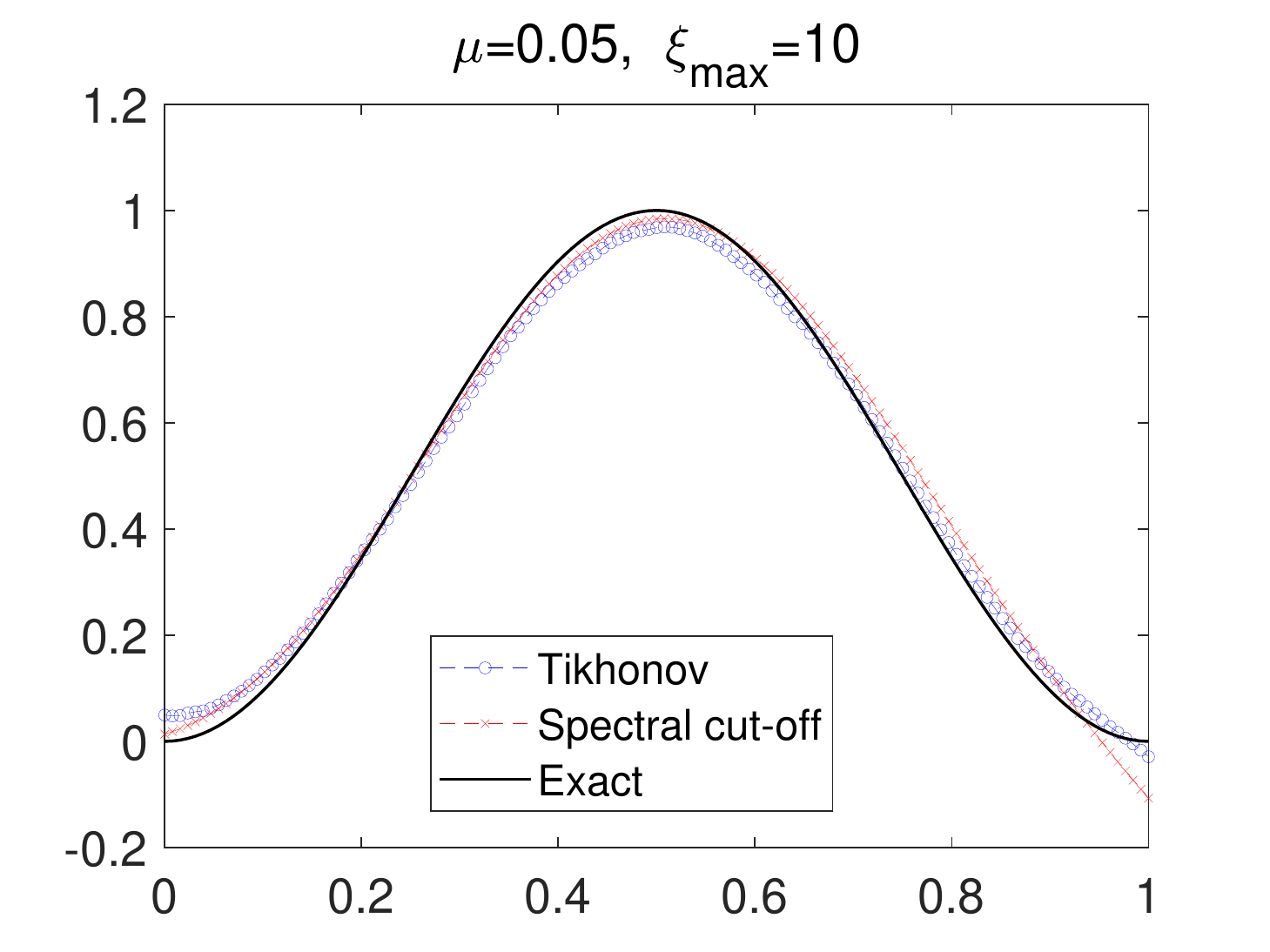}
  \includegraphics[width=0.32\textwidth]{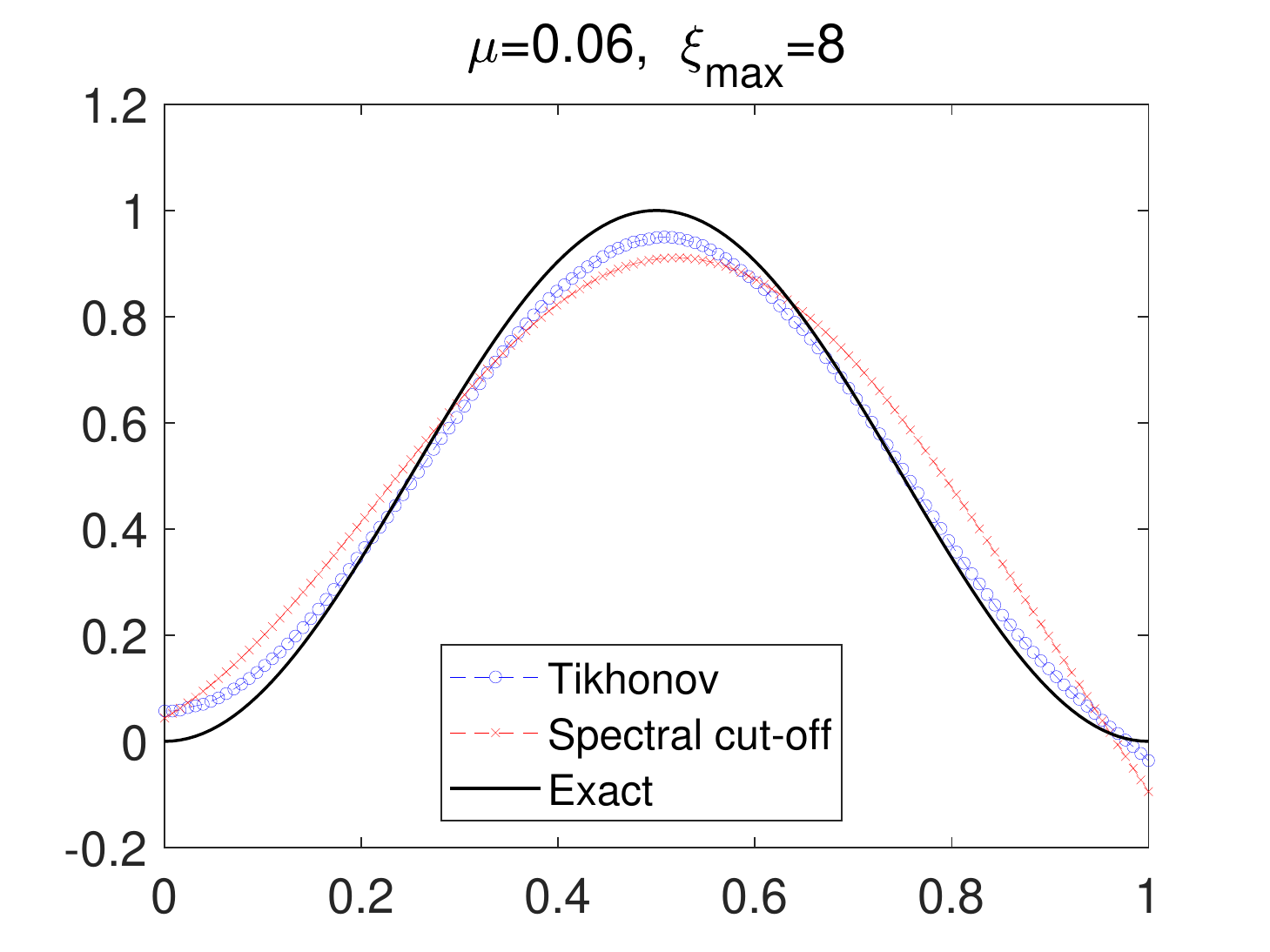}
  \caption{Example 1: the reconstruction of $q^2$ with different regularization parameters at a fixed noise level ($\epsilon=0.2$) and a fixed number of realizations ($P=10^6$).}\label{e102p6}
\end{figure}

\begin{figure}
  \includegraphics[width=0.32\textwidth]{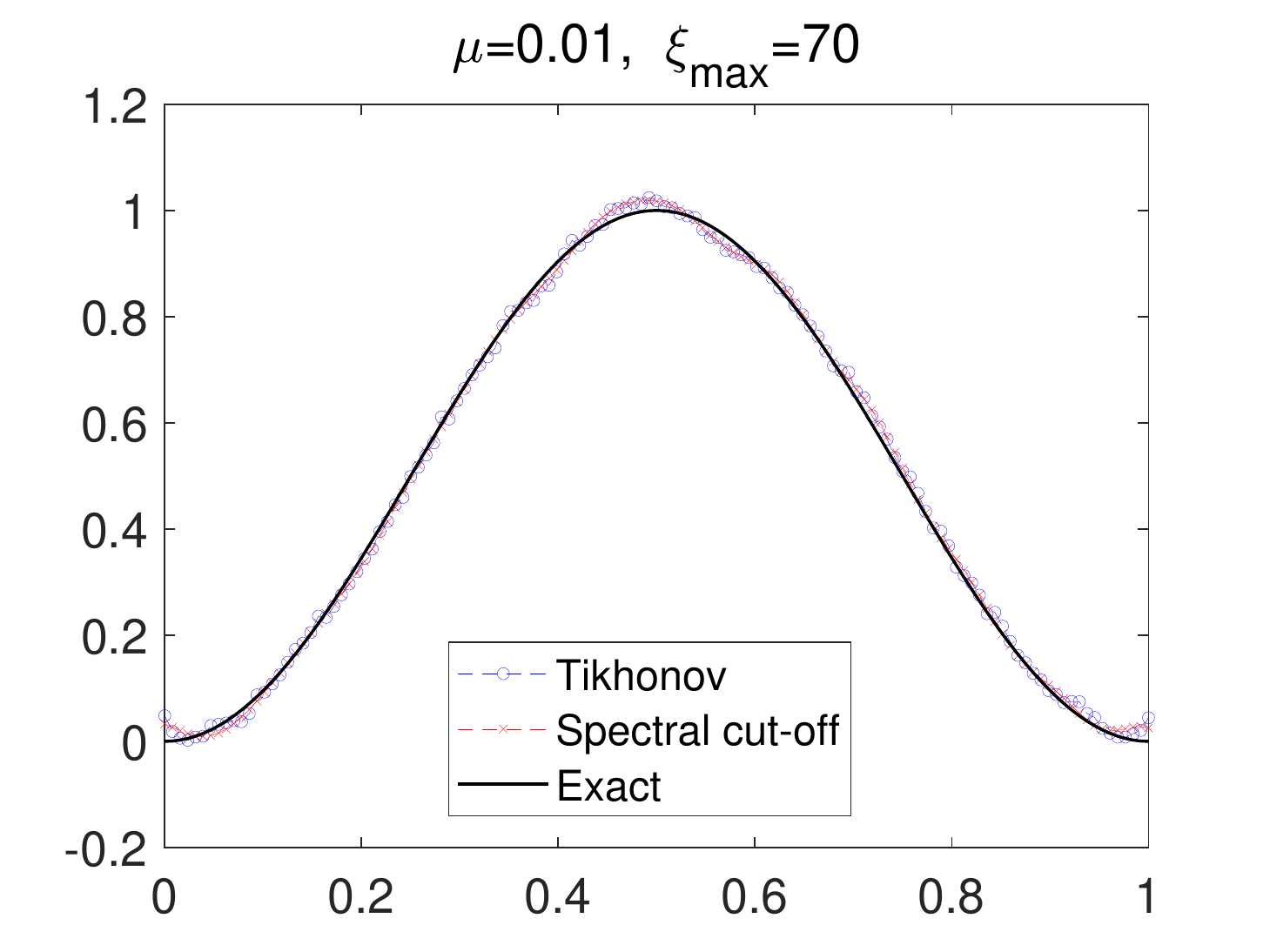}
  \includegraphics[width=0.32\textwidth]{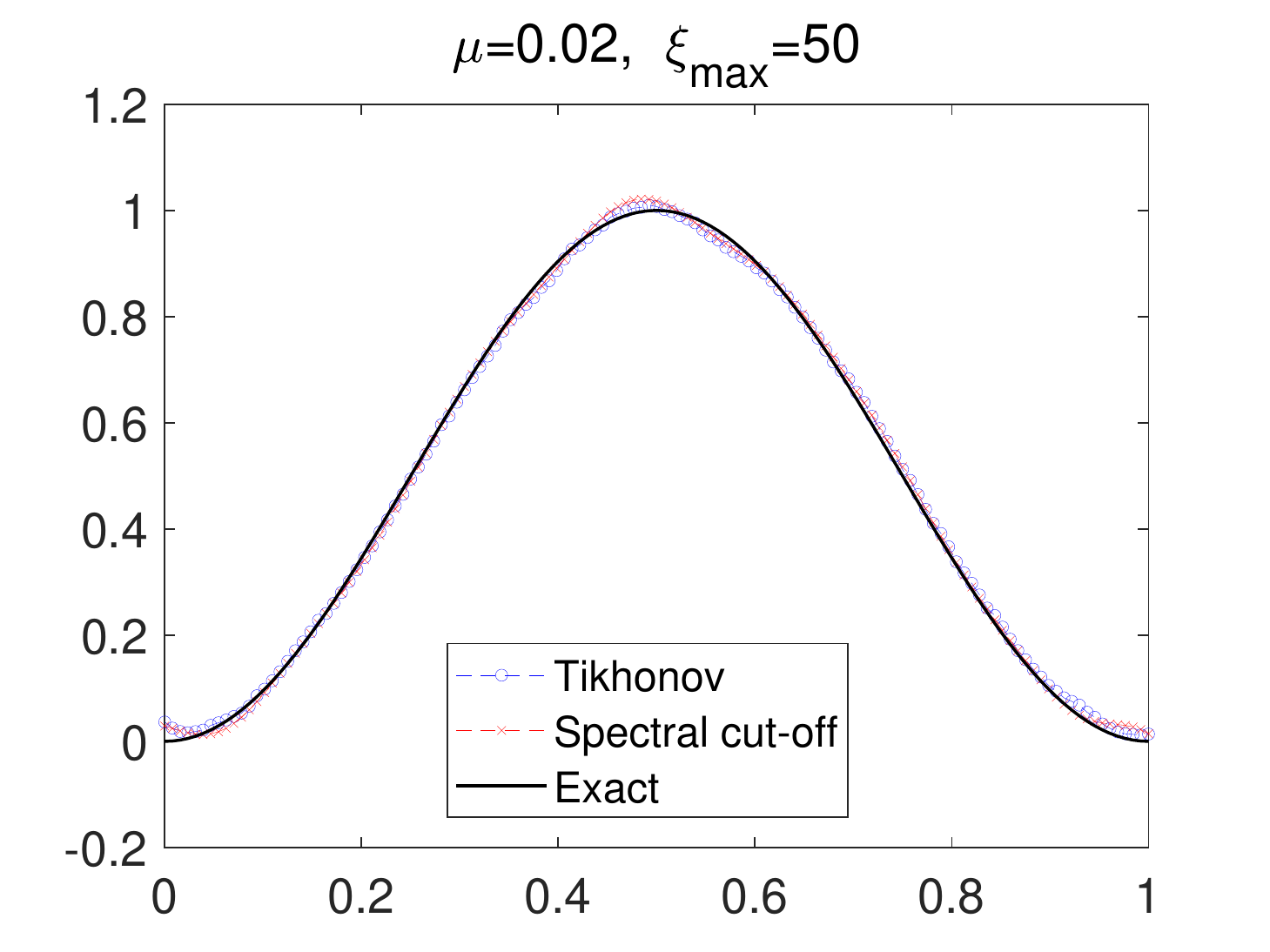}
  \includegraphics[width=0.32\textwidth]{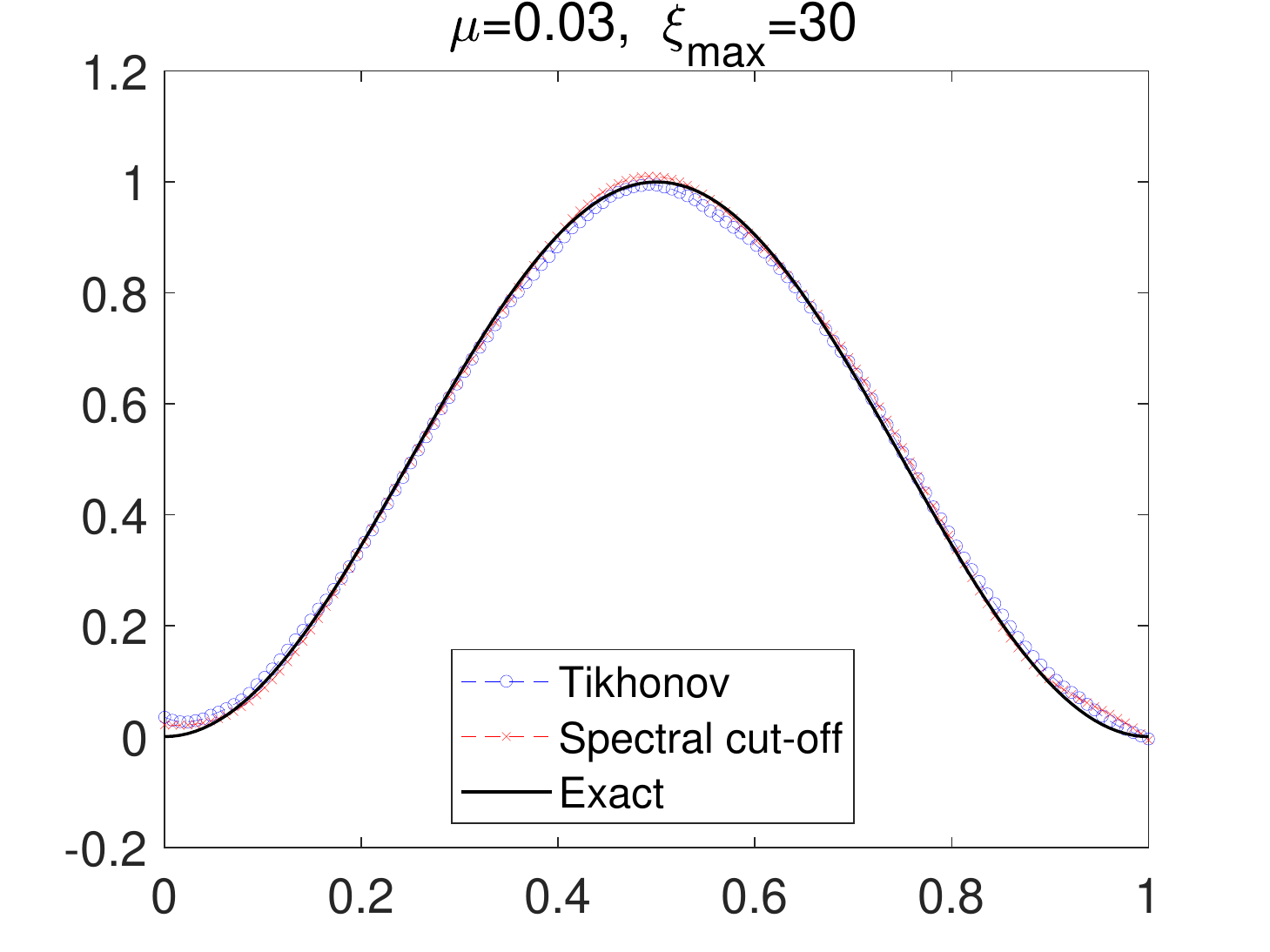}\\
  \includegraphics[width=0.32\textwidth]{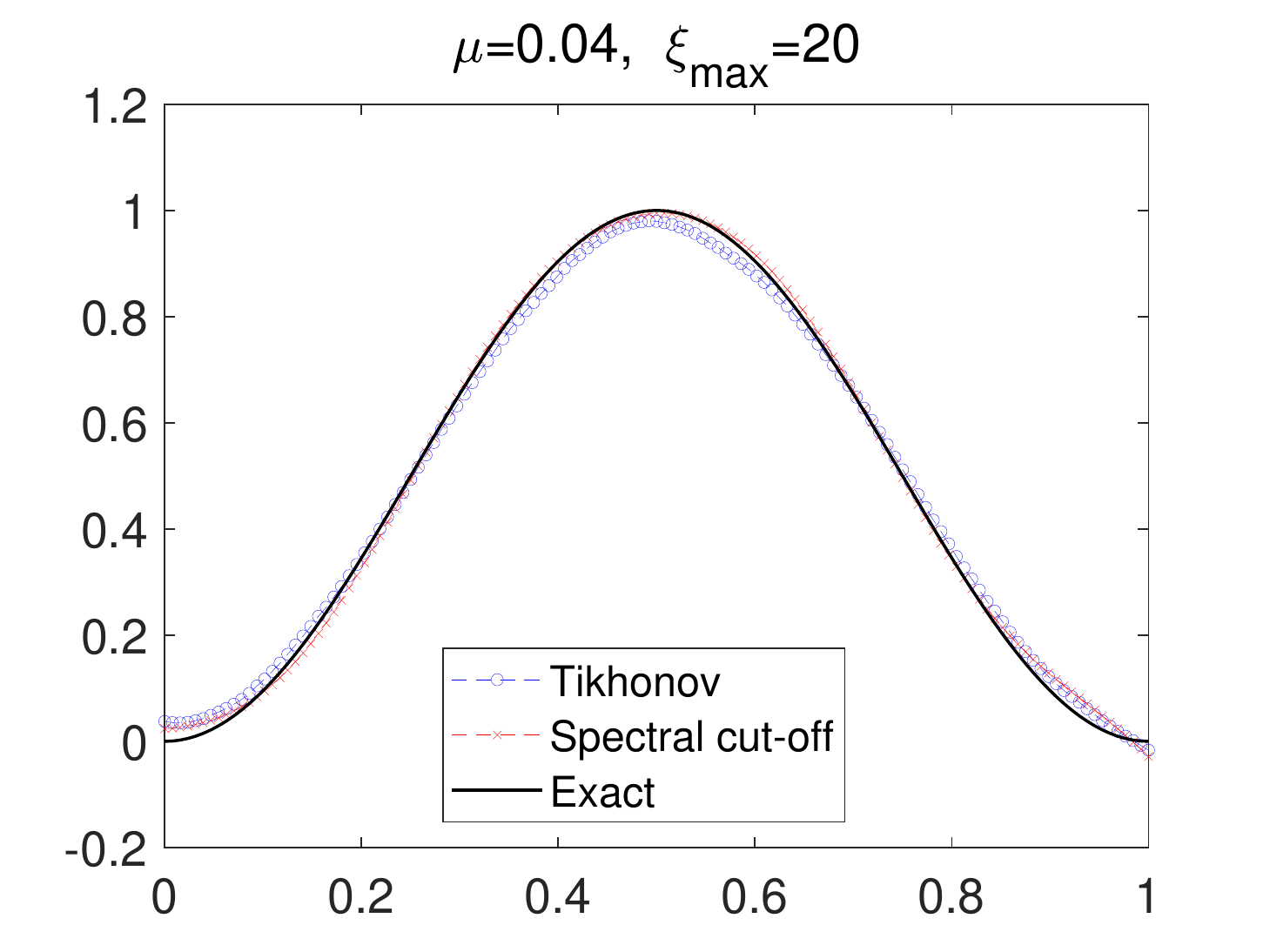}
  \includegraphics[width=0.32\textwidth]{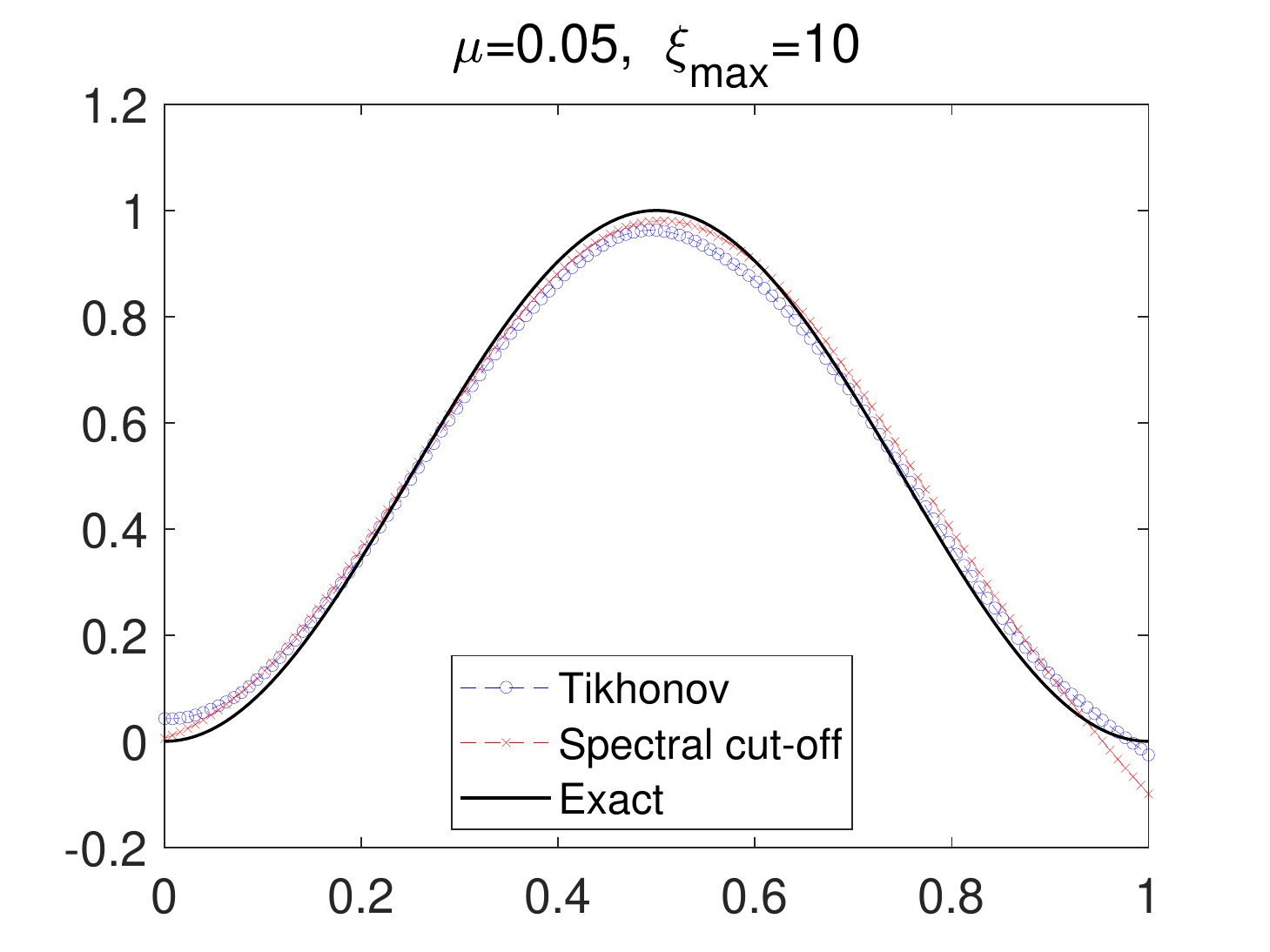}
  \includegraphics[width=0.32\textwidth]{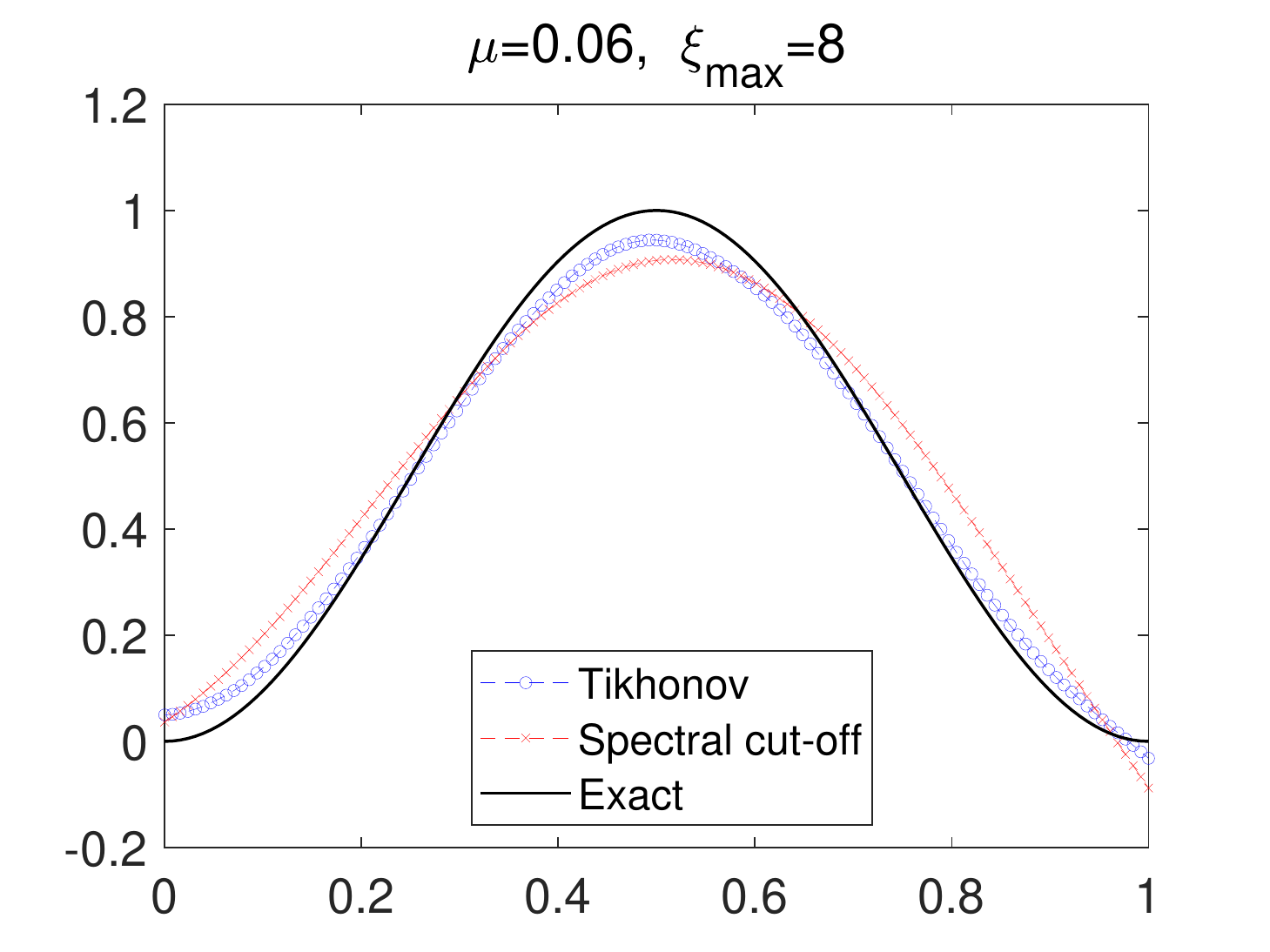}
  \caption{Example 1: the reconstruction of $q^2$ with different regularization parameters at a fixed noise level ($\epsilon=0.1$) and a fixed number of realizations ($P=10^6$).}\label{e101p6}
\end{figure}

\begin{figure}
  \includegraphics[width=0.32\textwidth]{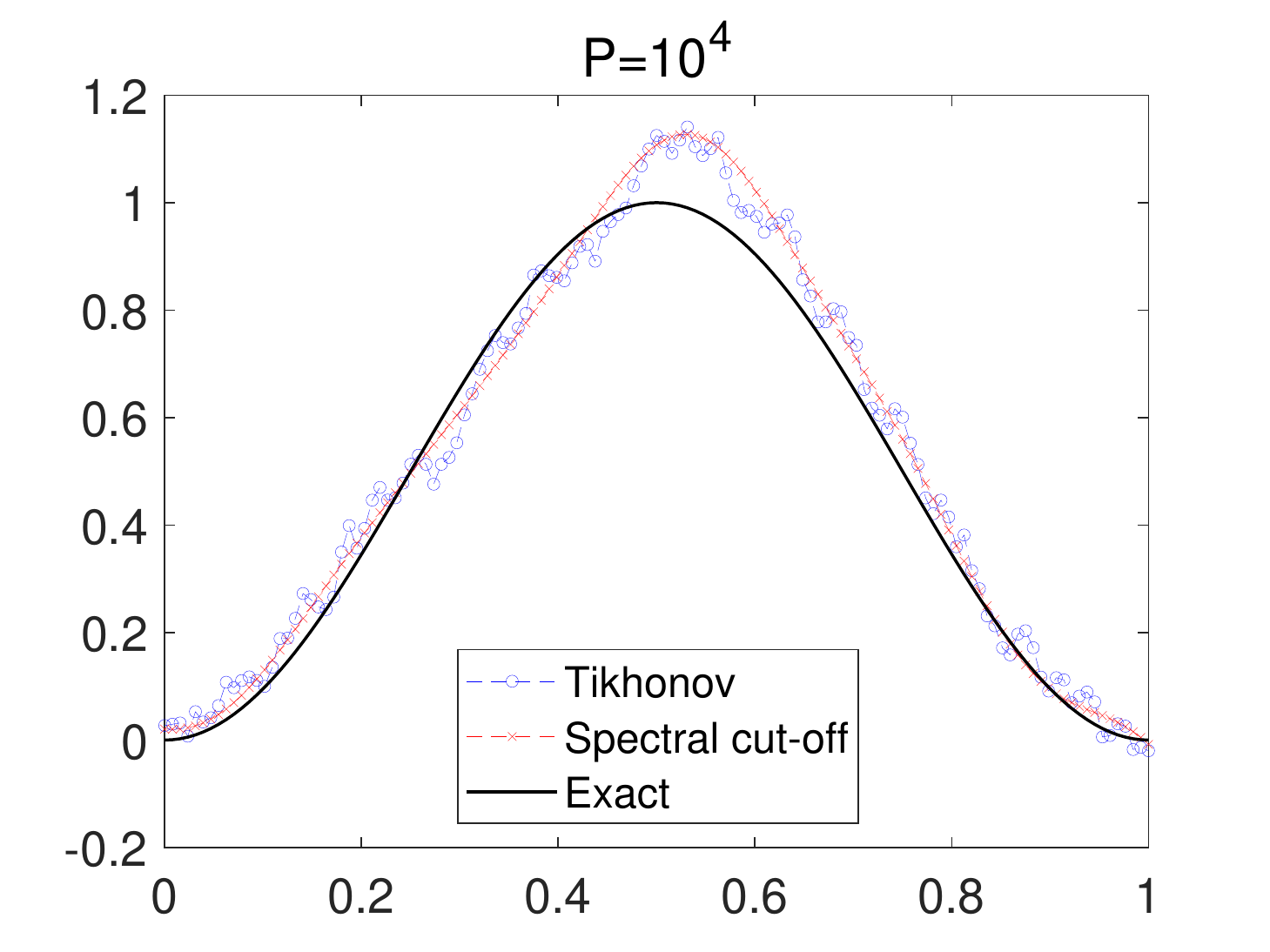}
  \includegraphics[width=0.32\textwidth]{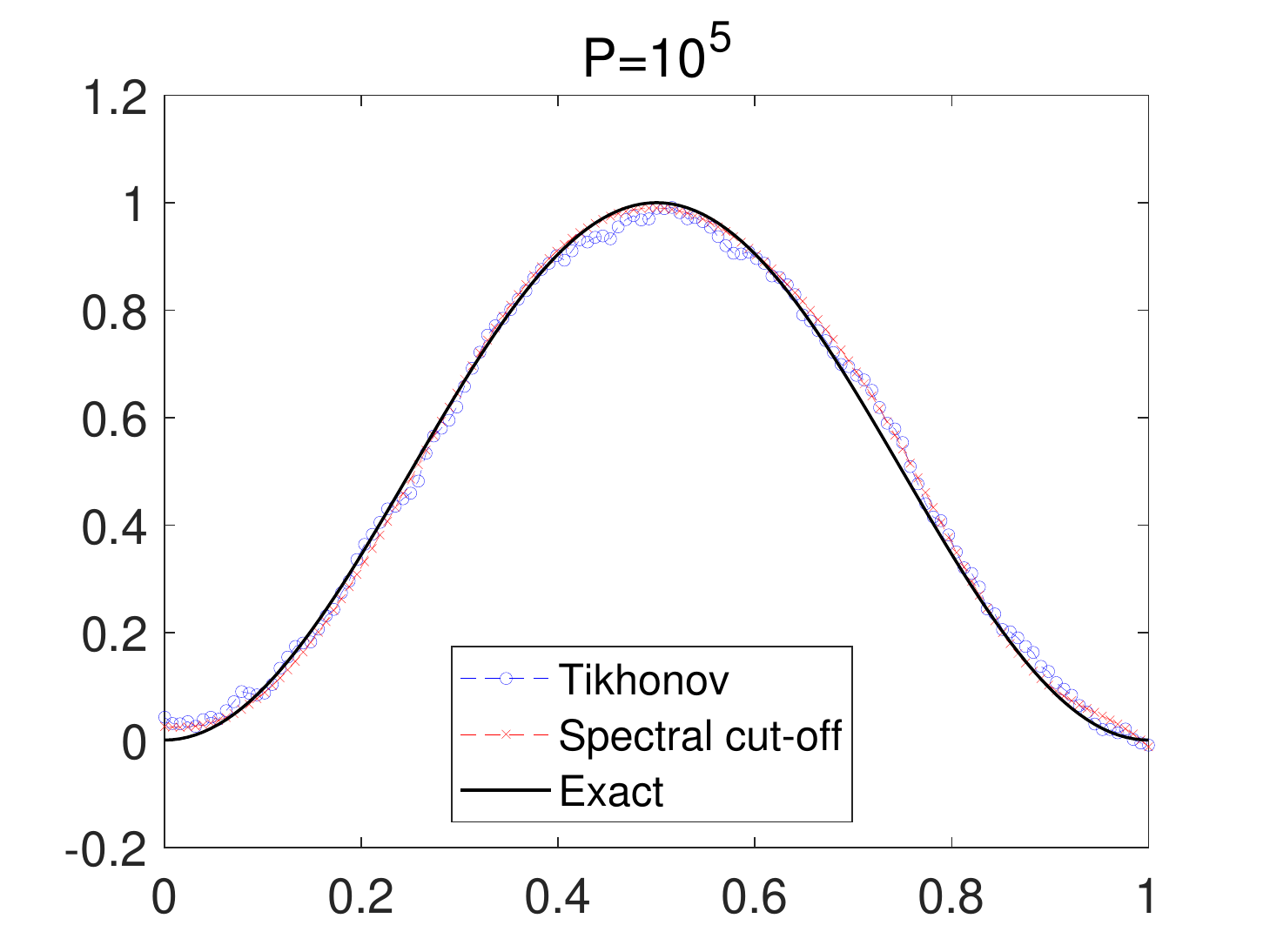}
  \includegraphics[width=0.32\textwidth]{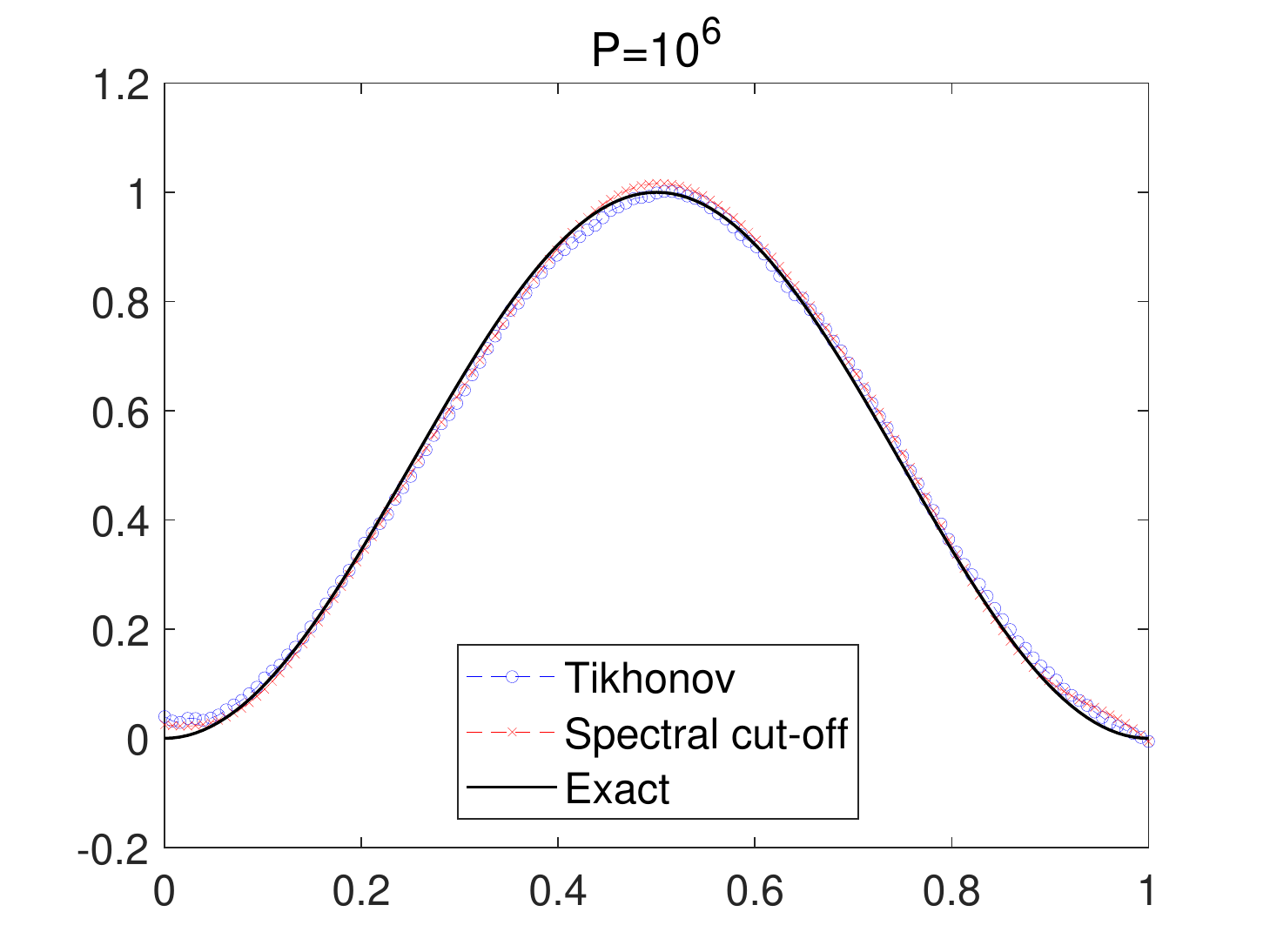}
  \caption{Example 1: the reconstruction of $q^2$ with a different number of realizations ($P=10^4,10^5,10^6$) at a fixed noise level ($\epsilon=0.2$) and a fixed regularization parameter ($\mu=0.03$, $\xi_{\rm max}=30$).}\label{e1p54}
\end{figure}

First, we report the data functions under the influence of the noise level $\epsilon$ and the number of realizations $P$. Figure \ref{e1psip6} shows the data function $\psi(t)$,  $t\in[0, 1]$ defined in \eqref{psifun} and its periodization $\Psi(t)$, $t\in[-1, 2]$ at different noise levels ($\epsilon=0.5, 0.2, 0.1$) with a fixed number of realizations ($P=10^6$). As expected, the data function $\psi$ is smoother if the noise level $\epsilon$ is smaller. Figure \ref{e1psip54} shows the periodized data function $\Psi(t)$, $t\in[-1, 2]$ with a different number of realizations ($P=10^4,10^5,10^6$) at a fixed noise level ($\epsilon=0.5$). It is clear to note that a larger number of realizations can produce a better approximation to the expectation in \eqref{psin} and thus yield a data function with less oscillation.

Next, we investigate the reconstructions by using different regularization parameters and noise levels at a fixed number of realizations. In Figures \ref{e105p6}--\ref{e101p6}, the exact function $q^2$ is plotted against the reconstructed results by using the Tikhonov and spectral cut-off regularization methods. In each figure, the results are shown for different regularization parameters ($\mu=0.01:0.01:0.06$ or $\xi_{\rm max}=70,50,30,20,10,8$) at a fixed noise level and a fixed number of realizations ($P=10^6$). It can be observed that the methods are stable and produce good reconstructions for appropriately chosen regularization parameters; the results are under-regularized for small $\mu$ or large $\xi_{\rm max}$, while the results are over-regularized for large $\mu$ or small $\xi_{\rm max}$; the methods work better for the regularization parameters given in the range $\mu\in[0.03, 0.04]$ or $\xi_{\rm max}\in[20, 30]$. Fixing the regularization parameters, we may compare the corresponding results shown in Figures \ref{e105p6}--\ref{e101p6} for different noise levels ($\epsilon=0.5, 0.2, 0.1$). Apparently, the results are better for smaller noise levels.

Finally, we consider the influence of the number of realizations. Figure \ref{e1p54} shows the results by using a different number of realizations ($P=10^4,10^5,10^6$) at a fixed noise level ($\epsilon=0.2$) and a fixed regularization parameter ($\mu=0.03, \xi_{\rm max}=30$). It can be seen that a larger number of realizations gives a better approximation to the expectation of the data and yields a better reconstruction.

\subsubsection{Example 2}

The exact potential function is
\begin{equation*}
q(t)=\left\{
\begin{aligned}
&0,&&\quad 0\leq t\leq\frac15,\\
&4\sqrt{t-\frac15}, &&\quad \frac15\leq t\leq\frac12,\\
&4\sqrt{-t+\frac45}, &&\quad\frac12\leq t\leq\frac45,\\
&0,&&\quad\frac45\leq t\leq1,
\end{aligned}
\right.
\end{equation*}
which is continuous on $[0, 1]$ but is not differentiable at $t=\frac{1}{5}, \frac{1}{2}, \frac{4}{5}$. We shall not document the detailed results with different parameters since the patterns are similar to those of Example 1.

\begin{figure}
  \includegraphics[width=0.32\textwidth]{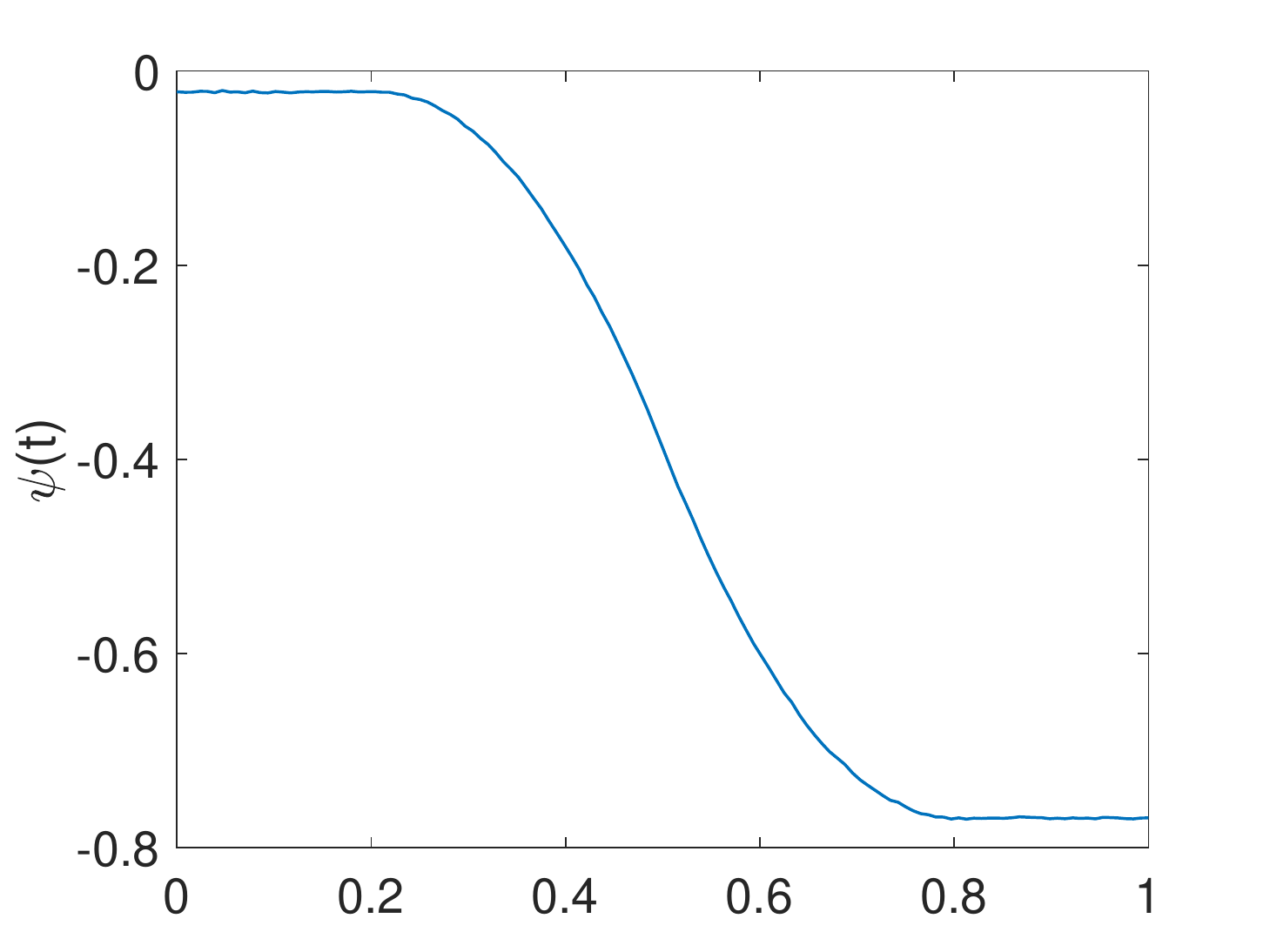}
  \includegraphics[width=0.32\textwidth]{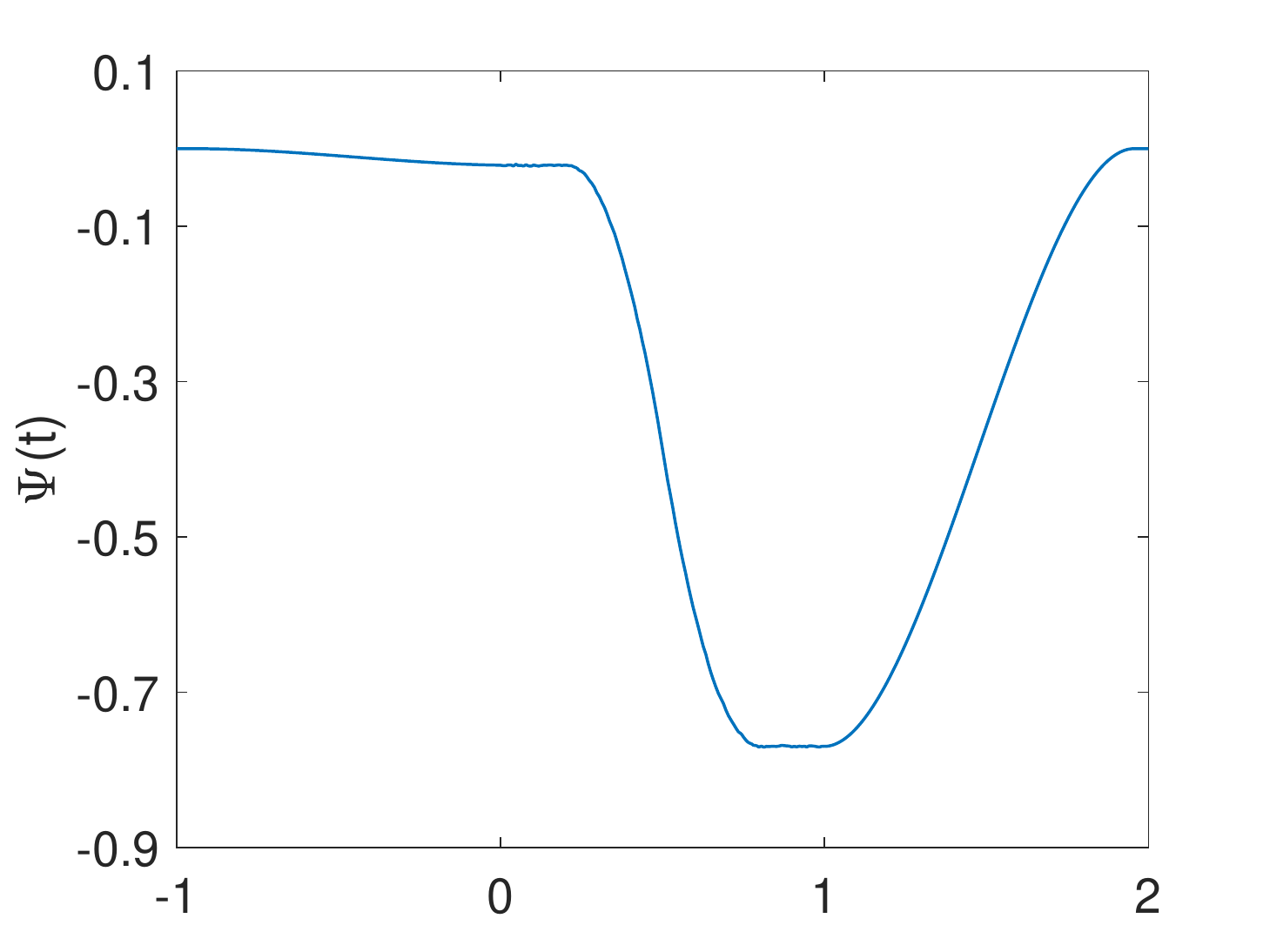}
  \includegraphics[width=0.32\textwidth]{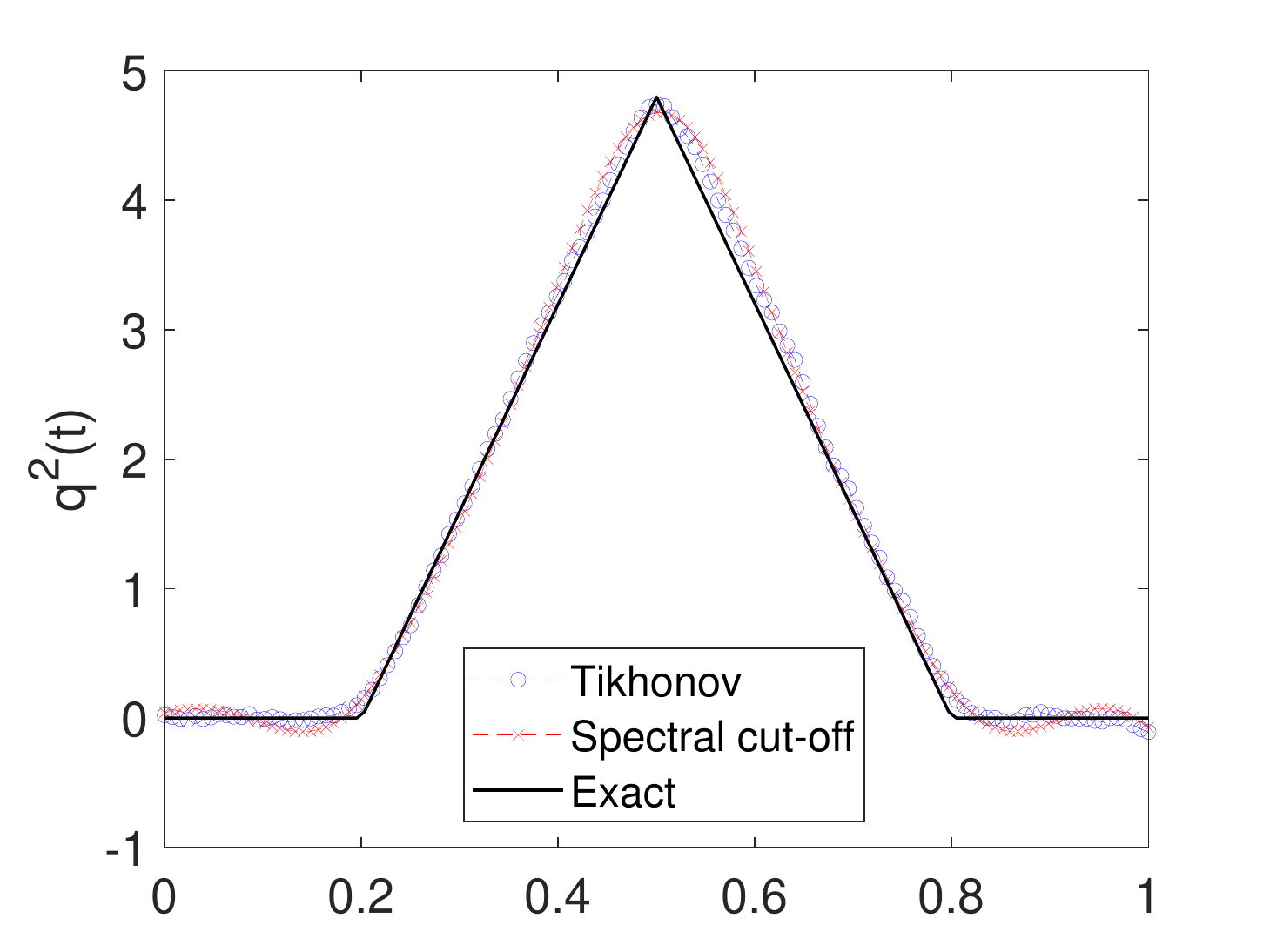}
  \caption{Example 2: the data function $\psi(t)$ on $[0, 1]$ (left), the periodization $\Psi$ on $[-1,2]$ (middle), and the reconstruction of  $q^2$ (right) with the parameters given by $P=10^5$, $\epsilon=0.2$, $\mu=0.02$, $\xi_{\rm max}=30$. }\label{e2P5}
\end{figure}

As suggested from the experiments of Example 1, the regularization parameters are chosen as $\mu=0.02$ and $\xi_{\rm max}=30$. Figure \ref{e2P5} shows the data function $\psi(t)$ for $t\in[0, 1]$, the periodization $\Psi(t)$ for $t\in[-1, 2]$, and the reconstruction of $q^2$ with the number of realizations $P=10^5$ and the noise level $\epsilon=0.2$. It can be seen that the Tikhonov regularization method and the spectral cut-off method give similar numerical results, and both methods work well for such a nonsmooth function.

\subsubsection{Example 3}

The exact potential is a piecewise constant function defined by
\begin{equation*}
q(t)=\left\{
\begin{aligned}
0,&&\quad 0\leq t\leq\frac15,\\
1, &&\quad \frac15<t\leq\frac12,\\
2, &&\quad\frac12<t\leq\frac45,\\
0,&&\quad\frac45< t\leq1.
\end{aligned}
\right.
\end{equation*}
As a discontinuous function, it contains infinitely many Fourier modes and the Fourier coefficients decay slowly. Thus this example is more difficult than the previous ones.

\begin{figure}
\includegraphics[width=0.32\textwidth]{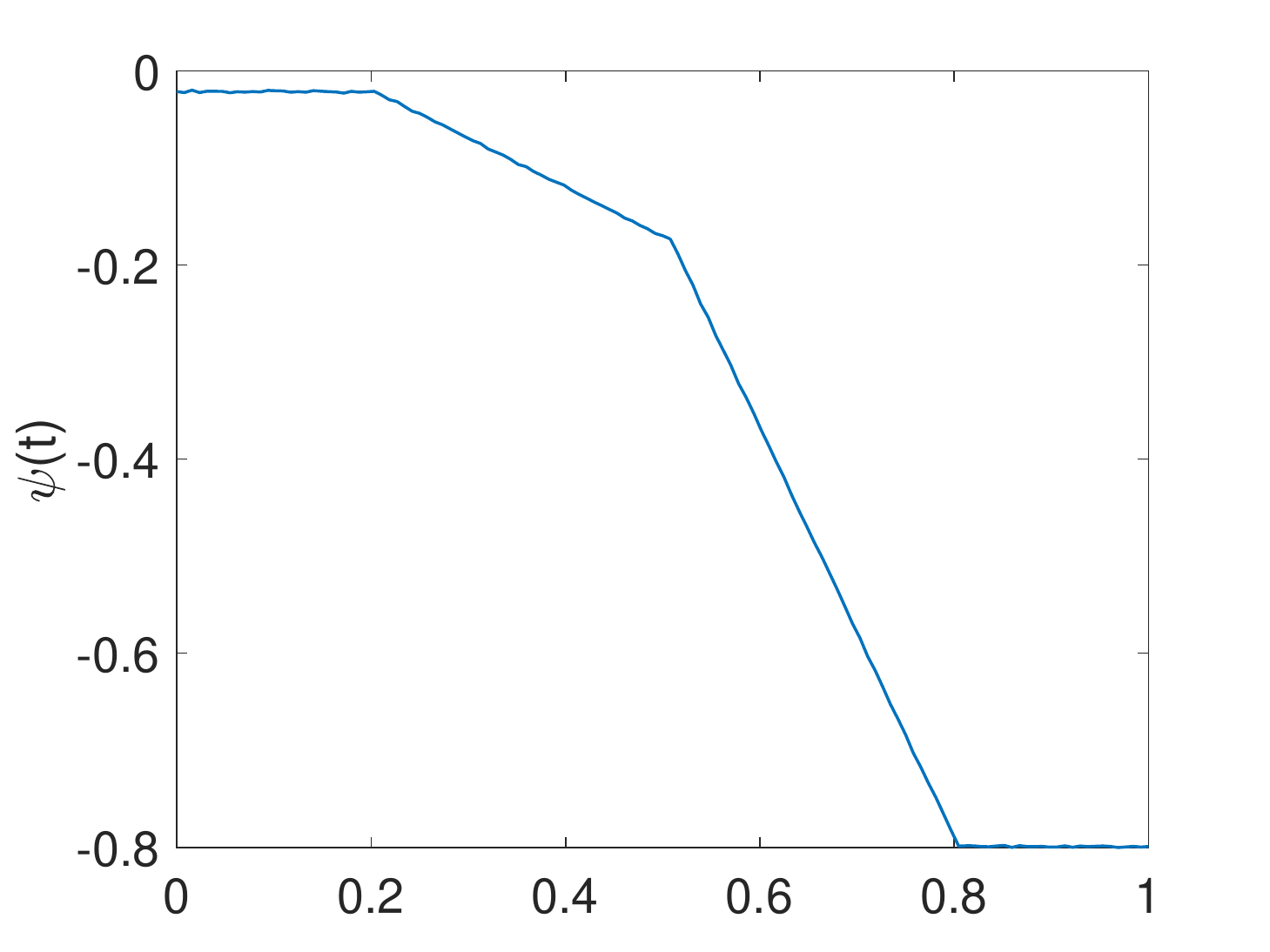}
  \includegraphics[width=0.32\textwidth]{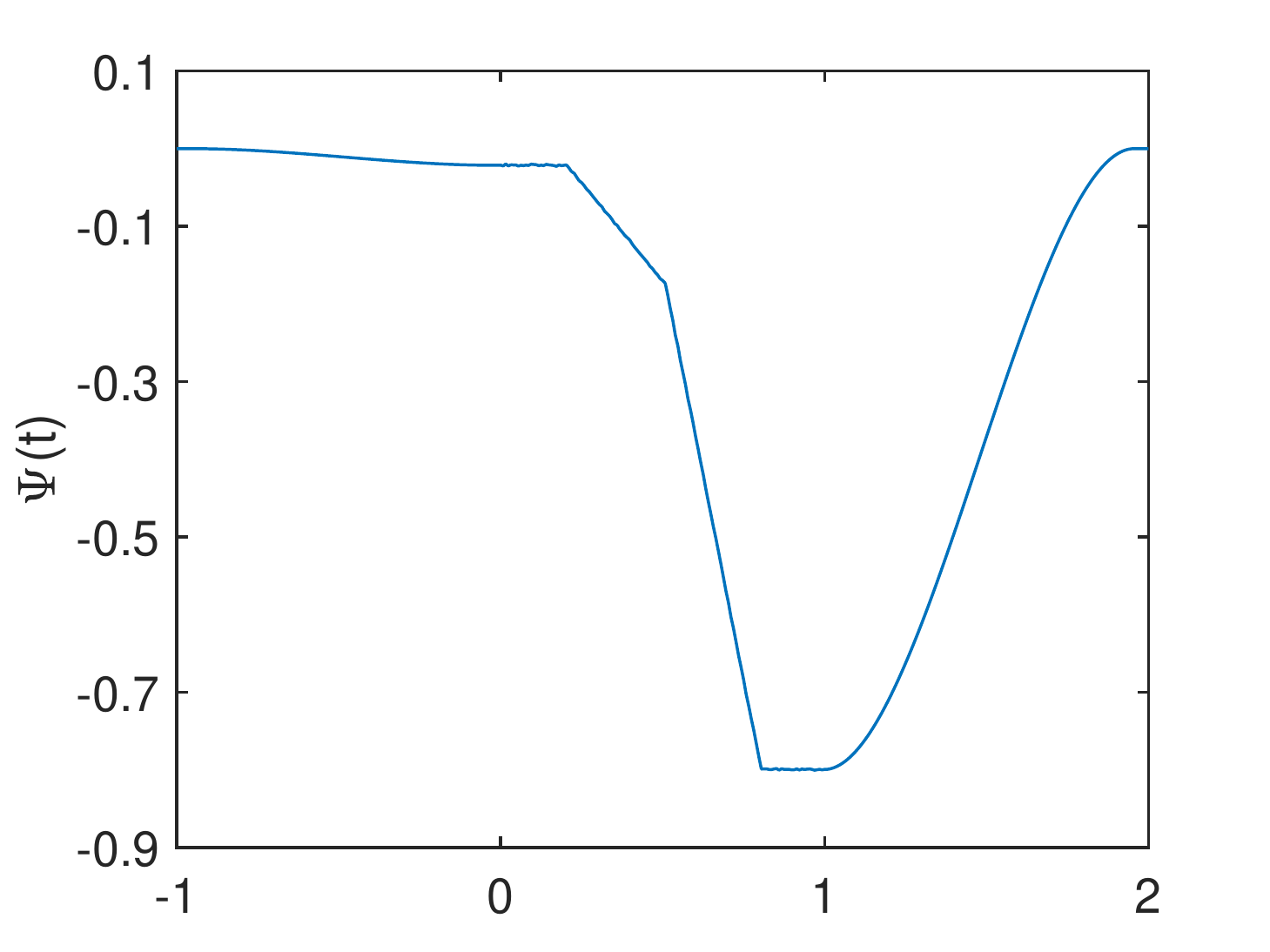}
  \includegraphics[width=0.32\textwidth]{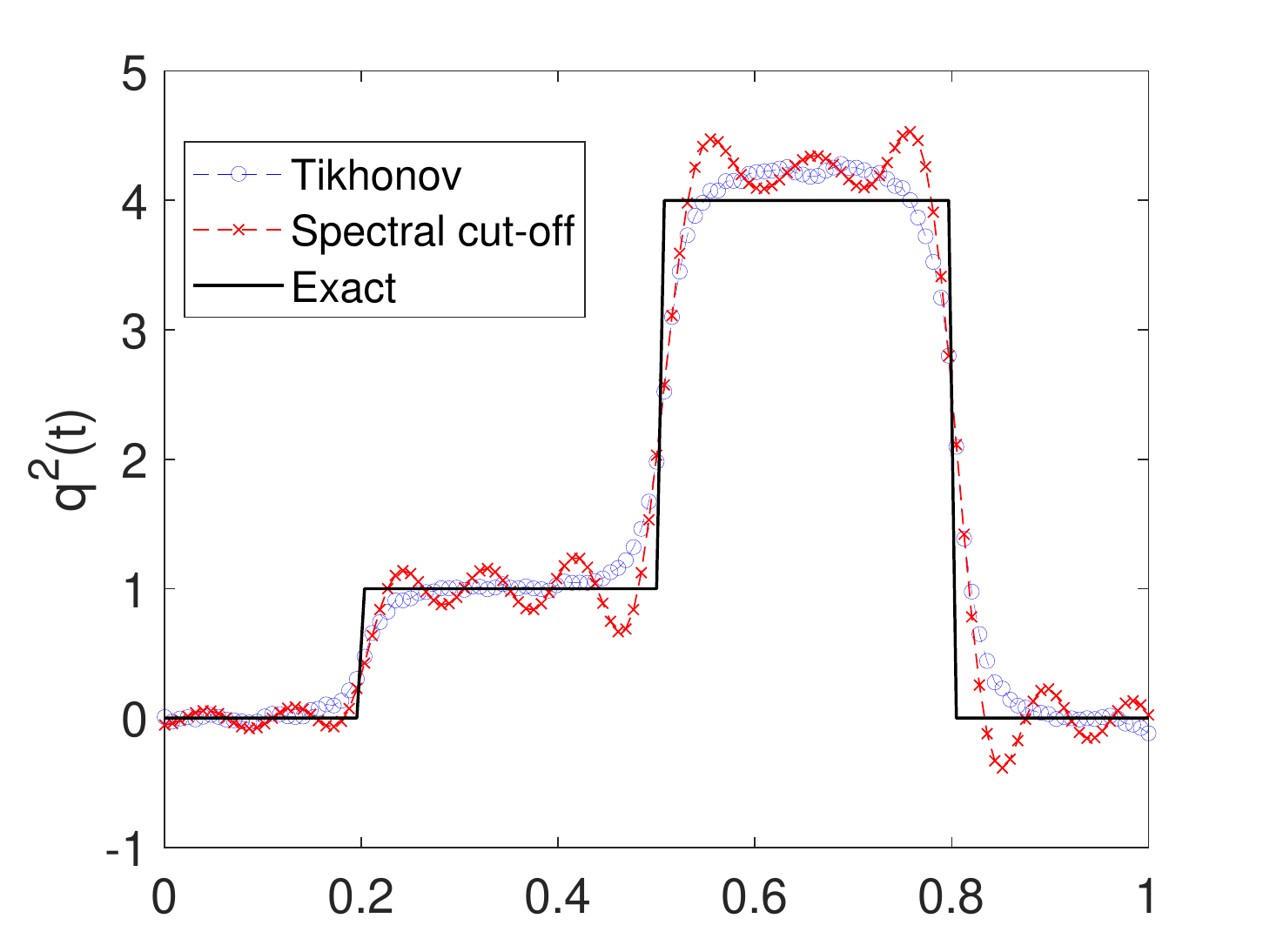}
  \caption{Example 3: the data function $\psi(t)$ on $[0, 1]$ (left), the periodization $\Psi$ on $[-1,2]$ (middle), and the reconstruction of $q^2$ (right) with the parameters given by $P=10^5$, $\epsilon=0.2$, $\mu=0.02$, $\xi_{\rm max}=70$.}\label{e3P5}
\end{figure}

Again, we shall not present the influence of the various parameters on the reconstructions but choose to show the results by using representative parameters. Figure \ref{e3P5} shows the data function $\psi(t)$ for $t\in[0, 1]$, the periodization $\Psi(t)$ for $t\in[-1, 2]$, and the reconstruction of $q^2$ with the regularization parameters $\mu=0.02$ and $\xi_{\rm max}=70$ for the fixed number of realizations $P=10^5$ and the fixed noise level $\epsilon=0.2$. It is worth mentioning that a larger frequency cut-off $\xi_{\rm max}$ should be chosen so that more Fourier modes of the potential function can be recovered for this discontinuous example. Comparing the two regularization methods, we observe that the Tikhonov regularization method yields a smoother reconstruction while the spectral cut-off method displays the Gibbs phenomenon for the reconstructed function, which is common for the Fourier based method to recover discontinuous functions.

\section{Conclusion}

We studied both the direct and inverse problems for the stochastic diffusion equation with a multiplicative time-dependent white noise. For the direct problem, we examined the existence, uniqueness, and regularity of the mild solution. For the inverse problem, an explicit reconstruction formula was deduced by establishing the relation between the deterministic diffusion equation and the stochastic diffusion equation. The uniqueness was obtained to determine $q^2$, which implies the uniqueness of the inverse problem for nonnegative potential functions. To overcome the ill-posedness of numerical differentiation, we adopted the Tikhonov and spectral cut-off regularization methods which were implemented efficiently by using the FFT. The results show that the methods are effective to reconstruct both smooth and nonsmooth potential functions.

In this work, the potential function is assumed to be time-dependent and the stochastic diffusion equation is driven by the multiplicative time-dependent white noise. We plan to consider the inverse problems for more general potential functions and other types of noise, such as the colored noise, the fractional Brownian motion, or the space-time white noise. The progress will be reported elsewhere in the future.


\end{document}